\documentclass[a4paper,12pt,reqno]{amsart}
\usepackage{latexsym, amsfonts, amsmath, amsthm, amssymb, enumerate, verbatim}
\usepackage[top=3cm, bottom=3cm, left=2.54cm, right=2.54cm]{geometry}

\renewcommand\Re{\mathop{{\rm Re}}}
\renewcommand\Im{\mathop{{\rm Im}}}
\newcommand{\Rb}{{\mathbb R}}

\newtheorem{thm}{Theorem}[section]
\newtheorem*{thm*}{Theorem}
\newtheorem{cor}[thm]{Corollary}
\newtheorem*{cor*}{Corollary}
\newtheorem{lem}[thm]{Lemma}
\newtheorem{prop}[thm]{Proposition}

\newtheorem*{con*}{Conjecture}
\newtheorem*{prob*}{Problem}

\theoremstyle{definition}
\newtheorem{defn}[thm]{Definition}
\theoremstyle{remark}
\newtheorem{rem}[thm]{Remark}

\newtheorem{ex}[thm]{Example}

\numberwithin{equation}{section}
\begin{document}

\noindent {\bf On Determinant Expansions for Hankel Operators}\par
Gordon Blower ${}^a$ and Yang Chen${}^b$\par
\vskip.1in
\noindent ${}^{a}$ \footnotesize{Corresponding author: g.blower@lancaster.ac.uk}\par
\footnotesize {Department of Mathematics and Statistics, Lancaster University,}\par
\footnotesize{ Lancaster, LA14YF, United Kingdom}\par
\vskip.05in
\footnotesize{$^{b}$Department of Mathematics, University of Macau,}\par
\footnotesize{Avenida da Universidade, Taipa, Macau, China}\par

\vskip.1in
17 January 2019\par
\vskip.1in

\noindent {\bf Abstract} Let $w$ be a semiclassical weight which is generic in Magnus's sense, and $(p_n)_{n=0}^\infty$ the corresponding sequence of orthogonal polynomials. The paper expresses the Christoffel--Darboux kernel 
as a sum of products of Hankel integral operators. For $\psi\in L^\infty (i{\mathbb R})$, let $W(\psi )$ be the Wiener-Hopf operator with symbol $\psi$. The paper gives sufficient conditions on $\psi$ such that $1/\det W(\psi )W(\psi^{-1})=\det (I-\Gamma_{\phi_1}\Gamma_{\phi_2})$ where $\Gamma_{\phi_1}$ and $\Gamma_{\phi_2}$ are Hankel operators that are Hilbert--Schmidt. For certain $\psi$, Barnes's integral leads to an expansion of  this determinant in terms of the generalised hypergeometric ${}_nF_m$. These results extend those of Basor and Chen [2], who obtained ${}_4F_3$ likewise. The paper includes examples where the Wiener--Hopf factors are found explicitly. \par
\par
\vskip.1in

{\bf Acknowledgements} Gordon Blower acknowledges the generous support as a Distinguished Visiting Scholar at the University of Macau. Yang Chen gratefully acknowledges the generous support of the Macau Science and Technology Development Fund under the grant numbers FDCT 130/2014/ A3 and FDCT 023/2017/A1 and the University of Macau through MYRG2014-00011-FST and MYRG2014-0004-FST.\par
\vskip.1in

\begin{section}{Introduction}

\begin{defn} (i) Let $\phi\in L^2(0, \infty )$. Then the Hankel operator with scattering function $\phi$ is the integral operator  
\begin{equation} \Gamma_\phi f(x)=\int_0^\infty \phi (x+y)\, f(y)\, dy\end{equation}
which is densely defined in $L^2(0, \infty )$. (The term scattering function is not to be confused with symbol function.)\par
\indent (ii) Let $\nu\in \ell^2({\mathbb N}\cup \{0\})$. Then the Hankel matrix corresponding to $\nu$ is $[\nu (j+k)]_{j,k=0}^\infty$, which gives a densely defined operator in $\ell^2({\mathbb N}\cup\{ 0\}).$\end{defn} 
\indent Given a trace class Hankel operator $\Gamma$, the spectrum consists of $0$ and a sequence of eigenvalues $\lambda_j$, listed according to algebraic multiplicity, such that $\sum_{j=0}^\infty \vert \lambda_j\vert$ converges. Then we define the Fredholm determinant of $\Gamma $ by $\det (I+\Gamma )=\prod_{j=0}^\infty (1+\lambda_j)$. For Hilbert--Schmidt $\Gamma$, we define the Carleman determinant by $\det_2(I+\Gamma )=\prod_{j=0}^\infty ( (1+\lambda_j)e^{-\lambda_j})$.The purpose of the present paper is to compute Fredholm determinants such as $\det (I+\Gamma_\phi )$,  using operator theory and tools from linear systems.\par

\indent The function 
\begin{equation}K(x)={\frac{\gamma\sin\pi x}{\pi\sinh \gamma x}}\qquad (\gamma >0)\end{equation}
\noindent is even, integrable and of rapid decay at infinity and has Fourier transform
\begin{equation}F(\xi ) ={\frac{\sinh (\pi^2/\gamma )}{\cosh (\pi^2/\gamma )+\cosh (\pi \xi/\gamma )}}.\end{equation}
The Wiener--Hopf operator on $L^2(0, \infty )$ is $W(F): f(x)\mapsto \int_0^\infty K(x-y) f(y)\, dy$. 
The Wiener--Hopf factorization 
\begin{equation}1-F(\xi )=\psi_- (i\xi )\psi_+(i\xi )\end{equation}
\noindent was considered by Basor and Chen [2], who obtained various identities for determinants of related Hankel operators on $L^2(0, \infty )$. The following integral plays a central role in their analysis
\begin{equation}\label{Barnes}\int_{-i\infty}^{i\infty} \Bigl( {\frac{\Gamma (a+s)\Gamma (1-a+s)\Gamma (b-s)\Gamma (1-b-s)}{\Gamma (b+s)\Gamma (1-b+s)\Gamma 
(a-s)\Gamma (1-a-s)}}-1\Bigr) z^s{\frac{ds}{2\pi i}},\end{equation}
\noindent where $\Gamma$ is Euler's gamma function, and $a$ and $b$ are real. Integrals of this form were used by Mellin, Barnes and Meier [11] page 49 to develop theories of special functions.

In section four, we introduce an algebra ${\mathcal C}_2$ of complex functions on a strip containing $i{\mathbb R}$ such that each invertible $\psi\in {\mathcal C}$ has a Wiener --Hopf factorization 
$\psi (i\xi )=\psi_- (i\xi )\psi_+ (i\xi )$, and we consider the Wiener--Hopf operator $W(\psi )$ of $L^2(0, \infty )$ with symbol $\psi$. Then in section 5, we consider the functions 
\begin{equation}\phi_1(x)= \int_{-\infty }^\infty \Bigl( {\frac{\psi_-(i\xi )}{\psi_+(i\xi )}}-1\Bigr)e^{-i\xi x}\, d\xi \qquad (x>0)\end{equation}
\begin{equation}\phi_2(x)=\int_{-\infty }^\infty \Bigl( {\frac{\psi_+(-i\xi )}{\psi_-(-i\xi )}}-1\Bigr) e^{-i\xi x}\, d\xi \qquad (x>0);\end{equation}
and the Hankel integral operators $\Gamma_{\phi_1}$ and $\Gamma_{\phi_2}$. The main Theorem 5.1 gives sufficient conditions for the formula
\begin{equation}1/\det W(\psi )W(\psi^{-1})=\det (I-\Gamma_{\phi_1}\Gamma_{\phi_2}),\end{equation}
 along with sufficient conditions for the Hankel operators to be self-adjoint. \par
\indent Self-adjoint bounded Hankel operators have been characterized up to unitary equivalence by the results of [20]. The methods of [20] emphasized the importance of linear systems, and in this paper, linear systems are used to obtain expansions of the Fredholm determinant $\det (I-\Gamma_{\phi_1}\Gamma_{\phi_2})$. In section 6, we consider Wiener--Hopf factorizations which lead to Barnes's integrals as in (\ref{Barnes}), so that $\phi_1$ and $\phi_2$ have explicit expansions in terms of exponential bases. When interpreted with suitable linear systems,  these formulas give expansions of $\det (I-\Gamma_{\phi_1}\Gamma_{\phi_2})$  in terms of the generalised hypergeometric ${}_nF_m$. These results extend those of Basor and Chen [2], who obtained ${}_4F_3$ likewise. In section 7, we make specific choices of $\psi$ and interpret our results in particular examples.\par
\indent 
\begin{ex}\label{ex1} In the theory of random Hermitian matrices, the following example arises frequently. Let $w_0(x)$ be a continuous, positive and integrable weight on $(0,b)$. Then we can take $Z_b>0$ such that
\begin{equation}\label{prob}Z_b^{-1}\prod_{1\leq j<k\leq n} (x_j-x_k)^2 \prod_{j=1}^n w_0(x_j)dx_j\end{equation}
 gives a probability measure on $(0, b)^n.$ In (\ref{Hankeldet}), we identify $Z_b$ with a Hankel determinant.

For a bounded and measurable function $f:{\mathbb R}\rightarrow {\mathbb C}$, we define the linear statistic $\sum_{j=1}^n f(x_j)$ and consider the exponential moment generating function
\begin{equation}\label{linstat} {\mathbb E} e^{-\sum f}={\frac{\int_{(0,b)^n} \exp (-\sum_{j=1}^n f(x_j)) \prod_{1\leq j<k\leq n} (x_j-x_k)^2\prod_{j=1}^n w_0(x_j)dx_j}{\int_{(0,b)^n} \prod_{1\leq j<k\leq n} (x_j-x_k)^2\prod_{j=1}^n w_0(x_j)dx_j}}.\end{equation}
\indent In particular, with $f(x)=-\log (\lambda -x)$, we have $p_n(\lambda )={\mathbb E}\prod_{j=1}^n (\lambda-x_j)$, which is a monic polynomial of degree $n$. Moreover, Heine [13] showed that $(p_n(\lambda ))_{n=0}^\infty $ is the sequence of monic orthogonal polynomials with respect to the weight $w_0$. We introduce $h_j=\int p_j(x)^2 w(x)\, dx.$ Then the Hankel determinant 
\begin{equation}\label{Hankeldet}D_n[w_0]=\det\Bigl[\int_{(0,b)}x^{j+k}w_0(x)\, dx\Bigr]_{j,k=0}^{n-1}\end{equation}
satisfies
\begin{equation}D_n[w_0]=\prod_{j=0}^{n-1} h_j.\end{equation}
and $Z_b=D_n[w_0]$. In section 3, we consider how Fredholm determinants are related to finite Hankel determinants $\det [\nu (j+k)]_{j,k=0}^{n-1}$ when the weight $w_0$ is semiclassical in Magnus's sense [19]. Our results continue the analysis by Tracy and Widom [32].\par

\end{ex}

\end{section}

 \begin{section}{Linear systems and associated Hankel operators}\par
The results of this section enable us to use linear system methods to compute Fredholm determinants of Hankel operators. For a complex separable Hilbert space $H$, we let ${\mathcal L}(H)={\mathcal L}^\infty (H)$ be the space of bounded linear operators on $H$ with $\Vert T\Vert$ the usual operator norm of $T\in {\mathcal L}(H)$, and ${\mathcal L}^1(H)$ the ideal of trace class operators; then for $1\leq p<\infty$, let ${\mathcal L}^p(H)$ be the ideal of operators such that the Schatten $p$-norm $\Vert T\Vert_{{\mathcal L}^p(H)}=\bigl( {\hbox{trace}} (T^\dagger T)^{p/2}\bigr)^{1/p}$ is finite.\par

\indent The Mellin transform $f^*(s)=\int_0^\infty x^{s-1} f(x)dx$ gives a unitary transformation $f(x)\mapsto f(i\xi +1/2)/\sqrt{2\pi}$ from $L^2(0, \infty)\rightarrow L^2(i{\mathbb R})$. Let ${\mathbb C}_+=\{ z\in {\mathbb C}: \Re z>0\}$ be the right half-plane and let $H^2({\mathbb C}_+)$ be the Hardy space of holomorphic functions $f$ on ${\mathbb C}_+$ such that $\sup_{x>0}\int_{-\infty}^\infty \vert f(x+i\xi )\vert^2\, d\xi $ is finite. By the Paley--Wiener Theorem, the Mellin transform gives a unitary transformation that restricts to the orthogonal subspaces

$L^2(0, \infty) =L^2(0,1)\oplus L^2(1, \infty )$ to $L^2(i{\mathbb R})=H^2({\mathbb C}_+)\oplus H^2( {\mathbb C}_-)$.\par
\indent Let $L_j(x)=(j!)^{-1}e^x(d/dx)^j(x^je^{-x})$ be the Laguerre polynomial of order $0$ and degree $j$; then $(e^{-x/2}L_j(x))_{j=0}^\infty$ gives an orthonormal basis of $L^2(0, \infty )$. Taking the Laplace transform of the $(e^{-x/2}L_j(x))_{j=0}^\infty$  , we obtain an orthonormal basis for the space $H^2({\mathbb C}_+)$, namely 
\begin{equation}\label{rationalbasis}\Bigl( {\frac{(\lambda -1/2)^n}{\sqrt{2\pi} (\lambda +1/2)^{n+1}}}\Bigr)_0^\infty. \end{equation} With ${\mathbb N}=\{ 1, 2, \dots \}$, we introduce the standard Hilbert sequence space $\ell^2({\mathbb N}\cup \{ 0\})$, with the standard orthonormal basis $(e_n)$ and introduce the usual shift operator by the operation
 $Se_n=e_{n+1}$ on $\ell^2({\mathbb N}\cup \{ 0\})$. There is an unitary map $H^2({\mathbb C}_+)\rightarrow \ell^2({\mathbb N}\cup \{ 0\})$. 
We have unitary maps between the Hilbert spaces
\begin{equation} \begin{matrix} L^2(0, 1)&\rightarrow& H^2({\mathbb C}_+)\\
\downarrow&\nearrow&\downarrow\\
L^2(0, \infty )&\rightarrow & \ell^2({\mathbb N}\cup \{0\})\end{matrix}\end{equation}
where the top arrow is the Mellin transform, the maps down on the left is the change of variables $x=e^{-\xi}$ for $0<x<1$ and $\xi>0$, and the bottom arrow across is the expansion in terms of the Laguerre basis. The diagonal arrow is the Laplace transform, and the right downward arrow is given by expansion with respect to (\ref{rationalbasis}).

There are several equivalent expressions for the Hilbert--Schmidt norm of Hankel operators that appears here. Suppose that $\phi_1, \phi_2\in L^2(0, \infty )$, and extend them to $L^2(-\infty , \infty )$ by letting $\phi_j(u)=0$ for all $u<0$. Then by a simple Fourier transform calculation as in [5]

\begin{equation}\int_{-2\alpha}^{2\alpha} \vert u\vert \phi_1(u)\phi_2(u)\, du=\int_{-\infty}^\infty \int_{-\infty}^\infty 4\Bigl({\frac{\sin 2(x-y)}{x-y}}-{\frac{\sin^2(x-y)}{(x-y)^2}}\Bigr) \hat\phi_1\bigl({\frac{x}{\alpha}}\bigr) \hat\phi_2\bigl({\frac{y}{\alpha}}\bigr)\ {\frac{dxdy}{(2\pi)^2}}.\end{equation}
Let $C^\infty_c$ be the space of infinitely differentiable functions that have compact support.

\begin{lem} (Basor, Tracy [6]) Suppose momentarily that $f\in C^\infty_c$ is real and even, so $f(x)=f(-x)$. Then the Mellin transform $f^*$ and the Fourier cosine transform $C(f)$ of the function $f$ satisfy
\begin{equation}{\frac{1}{\pi^2}}\int_{-\infty}^\infty \vert f^*(iy)\vert^2 y\tanh (\pi y)\, dy={\frac{1}{\pi^2}}\int_0^\infty x (C(f)(x))^2\, dx.\end{equation}\end{lem}

\begin{proof} The fractional derivative 
\begin{equation}D^{1/2} f(x)={\frac{1}{\sqrt{\pi}}} {\frac{d}{dx}}\int_x^\infty {\frac{f(u)}{\sqrt{u-x}}} du\end{equation}
\noindent has Mellin transform
\begin{equation}( D^{1/2}f)^*(s)=-{\frac{\Gamma (s)}{ \Gamma (s-1/2)}}f^*(s-1/2),\end{equation}
\noindent where $f^*(s)$ is the usual Mellin transform of $f$. Hence by the Plancherel formula for the Mellin transform 
\begin{align} \int_0^\infty \bigl( D^{1/2} f(x)\bigr)^2 dx&= {\frac{1}{2\pi i}}\int_{-i\infty}^{i\infty}  (D^{1/2}f)^*(s)( D^{1/2}f)^*(1-s)ds\nonumber\\
&=  {\frac{1}{2\pi i}}\int_{-i\infty}^{i\infty}  {\frac{\Gamma (s)\Gamma (1-s)}{\Gamma (s-1/2) \Gamma (1/2-s)}} f^*(s-1/2)f^*(1/2-s) ds\nonumber\\
&= {\frac{1}{2\pi}}\int_{-\infty}^\infty  f^*(iy)f^*(-iy)y\tanh (\pi y)\, dy.\end{align} 
\end{proof}

\begin{prop}\label{commutingdiagram}
The following is a commuting diagram of linear isometries, in which the top arrow is the Fourier cosine transform, and the left downwards arrow is the Mellin transform.
\begin{equation} \begin{matrix}\{ f: D^{1/2}f\in L^2(0, \infty )\}& \rightarrow &\{ \phi \in L^2((0, \infty ); xdx)\}\\
\downarrow &{}& \downarrow\\
\{ f^*\in L^2(i{\mathbb R}; y\tanh (\pi y)dy)\}&\rightarrow &\{ \Gamma_\phi\in {\mathcal L}^2\}\\
 \end{matrix} \end{equation}
\end{prop}

\indent We show that trace class Hankel operators on Hardy space $H^2({\mathbb C}_+)$ have a matrix representation with respect to reproducing kernels on the state space.
 Let ${\mathbb C}_+=\{ z\in {\mathbb C}: \Re z>0\}$ and  ${\mathbb C}_-=\{ z\in {\mathbb C}: \Re z<0\}$; then we introduce the usual Hardy spaces $H^2({\mathbb C}_+)$ and  $H^2({\mathbb C}_-)$ which are related by the unitary involution $J: H^2({\mathbb C}_+)\rightarrow H^2({\mathbb C}_-):$ $f(s)\mapsto f(-s)$. We regard $H^2({\mathbb C}_+)$ as a closed linear subspace of $L^2(i{\mathbb R}),$ and let $P_+: L^2(i{\mathbb  R})\rightarrow H^2({\mathbb C}_+)$ be the orthogonal projection. For $h\in L^\infty (i{\mathbb  R})$, let $M_h:L^2(i{\mathbb  R})\rightarrow L^2(i{\mathbb  R})$ be the multiplication operator $f\mapsto hf$.  The Laplace transform gives a unitary isometry ${\mathcal L}:L^2(0, \infty )\rightarrow H^2({\mathbb C}_+)$.\par
 \indent  Given $c\in L^\infty$, suppose that $\Gamma_c=P_+M_cJ$ is a bounded Hankel operator. Then by Nehari's and Fefferman's theorems [23], there exists $\psi\in L^\infty (i{\mathbb R})$ such that
\begin{equation}c(s)-c(\tau )=\int_{-\infty}^\infty \psi (i\omega )\Bigl( {\frac{1}{s-i\omega }}- {\frac{1}{\tau-i\omega }}\Bigr) {\frac{d\omega}{2\pi}}\qquad (s,\tau\in {\mathbb C}_+).\end{equation}
\noindent Note that $\psi$ determines $c$ up to an additive constant; adding a constant $\alpha$ to $c$ does not change $\psi$ of $\Gamma_c$. See [23]\par
\indent Let $H=H^2({\mathbb C}_+)$ be the state space and let ${\mathcal D}(A)=\{ g(s)\in H: s g(s )\in H\}$ with the graph norm.  Then we introduce the linear system $(-A,B,C)$  by 
\begin{align}\label{linsys3} A:{\mathcal D}(A)\rightarrow H: &\quad  g(s )\mapsto s g(s)\qquad (g\in {\mathcal D}(A));\nonumber\\
B: {\mathbb C}\rightarrow {\mathcal D}(A)^*:&\quad \beta\mapsto \beta \qquad (\beta\in {\mathbb C});\nonumber\\
C: {\mathcal D}(A)\rightarrow {\mathbb C}: &\quad g\mapsto {\frac{1}{2\pi}}\int_{-\infty }^\infty g(i\omega )\overline {c(i\omega )}\,d\omega \qquad (g\in {\mathcal D}(A)).\end{align}
\noindent The semigroup $(e^{-tA})_{t>0}$ operates by multiplication on the state space and is strongly continuous, so $e^{-tA}f(s)=e^{-st}f(s).$ 
Let $k_\zeta\in H^2({\mathbb C}_+)$ be the function $k_\zeta (s)=1/(s+\bar \zeta)$, so that $\langle f, k_\zeta \rangle =f(\zeta )$ for all $f\in H^2({\mathbb C}_+)$ and $\zeta\in {\mathbb C}_+$; one calls $k_\zeta (s)$ the reproducing kernel of $H^2({\mathbb C}_+)$. The various conjugates are introduced so that we can work with analytic, as opposed to anti-analytic, functions.\par
\vskip.05in
\begin{lem}Let $\Re\zeta_j>0$ and $c_j\in {\mathbb C}$ be such $\sum_{j=1}^\infty \vert c_j\vert \vert 1+\zeta_j\vert^2 /\Re \zeta_j$ is convergent.\par
\indent (i) Then the series $c(s)=\sum_{j=1}^\infty c_jk_{\zeta_j}(s)$ converges in $H^2({\mathbb C}_+)$ and $H^\infty$;\par
\indent (ii) the operators $\Gamma_c$ and $R_x=\int_x^\infty e^{-tA}BCe^{-tA}\, dt$ for $x>0$ are trace class on $H^2({\mathbb C}_+)$;\par
\indent (iii) $\Gamma_c^\dagger$ is unitarily equivalent to the Hankel integral operator $\Gamma_\phi$ on $L^2(0, \infty )$ with $\phi (t)=Ce^{-tA}B.$\end{lem}
\vskip.05in
\begin{proof}(i) The series $c(s)=\sum_{j=1}^\infty c_jk_{\zeta_j}(s)$ converges in $H^2({\mathbb C}_+)$ since $\sum_{j=1}^\infty\vert c_j\vert /\sqrt{\Re\zeta_j}$ converges. 
Also, $\Vert c\Vert_{L^\infty}\leq \sum_{j=1}^\infty \vert c_j\vert /\Re \zeta_j<\infty$, so $c\in H^\infty ({\mathbb C}_+)$; hence we can choose $\psi (i\omega) =c(i\omega )$ in the above, and deduce that $(c(z)-c(\alpha ))/(\alpha -z)$ belongs to $H^2({\mathbb C}_+)$ with norm $m/\sqrt{\Re \alpha }$. Hence by Lemma 2.2 of [23],  $Ce^{-tA}g\in L^2(0, \infty )$ for all $g\in {\mathcal D}(A)$ with 
\begin{equation}\int_0^\infty \vert Ce^{-tA}g\vert^2dt\leq {\frac{m\Vert g\Vert^2_H}{\Re \alpha }}\qquad (\Re \alpha >0, g\in {\mathcal D}(A)).\end{equation} 
\indent  (iii) One can easily check that $e^{-tA^*}k_\zeta (s)=e^{-t\bar\zeta }k_\zeta (s)$, hence 
\begin{equation}e^{-tA^*} c(s)=\sum_{j=1}^\infty c_je^{-t\bar\zeta_j} k_{\zeta_j}(s).\end{equation}
\noindent We introduce 
\begin{equation}\label{rankone}\phi (t)=Ce^{-tA}B=\langle 1, e^{-tA^*}c\rangle =\sum_{j=1}^\infty \bar c_j e^{-t\zeta_j}.\end{equation}

\noindent From the expansion of $\phi (t+u)$ as a series of rank one kernels $e^{-\zeta_j(u+t)}$, we deduce that $\Gamma_\phi$ is trace class with $\Vert\Gamma_\phi\Vert_{{\mathcal L}^1}\leq\sum_j\vert c_j\vert/(2\Re \zeta_j).$ One then checks that 
\begin{equation}\langle \Gamma_\phi^\dagger f,g\rangle_{L^2(0, \infty )} =\langle \Gamma_c {\mathcal L}f, {\mathcal L}g\rangle_{H^2};\end{equation}
\noindent the simplest way to do this is by selecting $f(x)=e^{-\xi x}$ and $g(x)=e^{-\zeta x}$, so that 
\begin{equation}\langle \Gamma_\phi^\dagger f,g\rangle_{L^2(0, \infty )}=\sum_{j=1}^\infty {\frac{c_j}{(\bar\zeta_j +\xi )(\bar\xi +\bar\zeta_j)}}=\Bigl\langle {\frac{c(s)-c(\xi )}{\xi -s}}, {\frac{1}{\zeta +s}}\Bigr\rangle_{H^2}.\end{equation}
\noindent Also, we deduce that
\begin{equation}\psi (i\xi )=\int_{0}^\infty e^{-i\xi t} \bar\phi (t)\, dt,\end{equation} 
\begin{equation}\label{transform}\bar\phi (t)=\int_{-\infty}^\infty \psi (i\xi )e^{i\xi t}{\frac{d\xi}{2\pi}}\qquad (t>0).\end{equation}
\indent (ii) Hence we can write
\begin{align}R_xf(z)&=\int_x^\infty e^{-tA}BCe^{-tA}f(z)\, dt\nonumber\\
&=\sum_{j=1}^\infty {\frac{\bar c_j e^{-x(z+\zeta_j)}f(\zeta_j)}{z+\zeta_j}}\nonumber\\
&=\sum_{j=1}^\infty \bar c_j e^{-xA}k_{\bar \zeta_j}(z)\langle f, e^{-xA^*}k_{\zeta_j}\rangle\end{align}
\noindent so $R_x\in {\mathcal L}^1(H)$. Hence $\Gamma_\phi$ and $\Gamma_c$ are trace class operators. \par
\indent Alternatively, one can introduce the sequence of $\lambda_j=(1-\zeta_j)/(1+\zeta_j)$ which satisfies $\vert \lambda_j\vert<1$ and $\sum_{j=1}^\infty\vert c_j\vert/(1-\vert \lambda_j\vert)<\infty$. Then one can show that $\Gamma_c$ is unitarily equivalent to a trace-class Hankel operator on the Hardy space $H^2$ of the unit disc, by Peller's criterion [24, page 232]. Incidently, Peller's criterion is sharp.\par
\end{proof}
\indent Any bounded Hankel integral operator generates a sequence of moments, in the following sense. For $\phi\in L^2(0, \infty )$, let $\Gamma_\phi$ be the Hankel integral operator and introduce the moment sequence
\begin{align}\mu_n&=\int_0^\infty \phi (x) L_n(x) e^{-x/2}\, dx\nonumber\\
&={\frac{1}{2\pi}}\int_{-\infty}^\infty \hat \phi (\xi ){\frac{(i\xi -1/2)^n}{(i\xi +1/2)^{n+1}}} d\xi\nonumber\\
&={\frac{1}{2\pi i}}\int_{\vert z\vert =1} \hat\phi \Bigl( {\frac{1+z}{2i (1-z)}}\Bigr)z^n{\frac{dz}{1-z}}\qquad (n=0, 1, \dots ).\end{align} 
Magnus has characterized the moment sequences that arise as $(\mu_n =\int_S x^n w(x)\, dx)$ for a semi classical weight on some subset of ${\mathbb C}\cup\{ \infty \}$, as we discuss in the next section. \par

\end{section}

\begin{section} {\bf From orthogonal polynomials to Hankel determinants}\par
\vskip.05in
\noindent Let $(p_n(x))_{n=0}^\infty $ be the sequence of monic orthogonal polynomials of degree $n$ for some continuous and positive weight $w_0(x)$ on $(0,b)$, given by the recurrence relation
\begin{equation}\label{polyrecurrence}xp_n(x)=p_{n+1}(x)+\alpha_n p_n(x)+\beta_n p_{n-1}(x).\end{equation}
Let $\int p_n(x)^2w_0(x)dx=h_n$; then $\beta_n-=h_n/h_{n-1}>0$. 
Let $Q_n$ be the orthogonal projection of $L^2(0,b)$ onto 
$${\hbox{span}}\{ \sqrt{w_0(x)} p_j(x); j=0, \dots, n-1\}.$$
Then the Christoffel--Darboux formula gives
\begin{equation}Q_n(x,y)={\frac{1}{h_{n-1}}}\sqrt{w_0(x)w_0(y)}{\frac{p_n(x)p_{n-1}(y)-p_n(y)p_{n-1}(x)}{x-y}}\end{equation}
so that $Q_n$ is an integrable operator. We show also that for suitable weights, $Q_n$ is a sum of products of Hankel operators.\par

\begin{defn}(Magnus, [19]) (i) Let $F(z)=\int (z-x)^{-1}w_0(x)\, dx$ be the Cauchy transform of the weight $w_0$ on $E=(0,b)$. The weight is said to be semi-classical if there exist polynomials $U,V,W$ with $W\neq 0$ such that
\begin{equation}W(z)F'(z)=2V(z)F(z)+U(z)\qquad (z\in {\mathbb C}\setminus {\mathbb R}).\end{equation}
Equivalently, the moments $\mu_k=\int x^kw_0(x)dx$ satisfy a recurrence relation
\begin{equation}\sum_{j=0}^m (\nu\xi_j+\eta_j)\mu_{j+\nu}=0\qquad (\nu=0,1 \dots ).\end{equation}
for some $\xi_j, \eta_k\in {\mathbb C}$ given by the coefficients of $V,W$, where $m$ is the maximum of the degrees of $V$ and $W$.\par
(ii) A pair of polynomials $(2V,W)$ is said to be generic if $W$ has degree $m$ where $m\geq 2$, the degree of $V$ is less than $m$, $W$ has $m$ simple zeros $\alpha_j$ and $2V/W$  has all  residues $2V(\alpha_j)/W'(\alpha_j)$ that are not integers. \par
(iii) Let $\vartheta (x)=0$ for $x<0$ and $\vartheta (x)=1$ for $x\geq 0$. 
\end{defn} 

\begin{thm}\label{Hankelfactors}{Let $w_0$ be a positive and continuous semiclassical weight on $[0,\infty )$ that corresponds to a generic pair $(2V,W)$.\par
(i) Then there exist $\phi_j, \psi_j\in L^2(0, \infty )$ such that
\begin{equation} Q_n(x,y)=\sum_{j=1}^N \int_0^\infty \phi_j(x+t)\psi_j(t+y)\, dt.\end{equation}
(ii) There exist scattering functions $\Phi, \Psi\in L^2((0, \infty ); {\mathbb C}^N)$ such that, for all $f\in L^\infty ({\mathbb R})$ as in (\ref{linstat}), 
\begin{equation} {\mathbb E} e^{-\sum f}=\det (I-M_{1-e^{-f}}\Gamma^T_\Psi \Gamma_\Phi ),
\end{equation}
where ${}^T$ denotes transpose.\par
(iii) For $f(x)=\beta\vartheta (x-t)$ with $\Re \beta>0$ and $\lambda={1-e^{-\beta}}$, the moment generating function of the random variable $\sharp\{j: x_j>t\}$ subject to the  probability (\ref{prob}) is given by 
\begin{equation} {\mathbb E} e^{-\sum \vartheta (.-t)}=\det (I-\lambda\Gamma^T_{\Psi_t} \Gamma_{\Phi_t} )
\end{equation}
where the scattering functions are shifted to $\Phi_t(x)=\Phi (x+t)$ and $\Psi_t(x)=\Psi (x+t)$.
}\end{thm}

\begin{proof} (i) Magnus [19] shows that for each such polynomial pair, there exists a weight $w_0$ with Cauchy transform $F$ and a polynomial $U$ such that $WF'=2VF+U$. From (11) of [19], we have $Ww'_0=2Vw_0$. Then by (17) of [19], there exist polynomials $\Omega_n$ and $\Theta_n$, and recursion coefficients $a_n$ such that with the matrices
\begin{equation}\label{iso}Y_n(x)=\begin{bmatrix} \sqrt{w_0(x)}p_n(x)\\ \sqrt{w_0(x)}p_{n-1}(x)\end{bmatrix}; J=\begin{bmatrix} 0&-1\\ 1&0\end{bmatrix}; 
A_n(x)={\frac{1}{W(x)}}
\begin{bmatrix} \Omega_n(x)&-a_n\Theta_n(x)\\ a_n\Theta_{n-1}(x)&-\Omega_n(x)\end{bmatrix}\end{equation}
we have an ordinary differential equation

\begin{equation}\label{ode}{\frac{d}{dx}}Y_n(x) =A_n(x)Y_n(x)\end{equation}
\noindent where the coefficient matrix $A_n(x)$ is rational with trace zero. The three-term recurrence relation (\ref{polyrecurrence}) for $p_n$  gives a positive sequence $(\beta_n)$ and a real sequence $(\alpha_n)$ such that
\begin{equation}Y_{n+1}=\begin{bmatrix}x-\alpha_{n}& -\beta_{n}\\ 1&0\end{bmatrix} Y_n,\end{equation}
so we have a recurrence relation for the matrices in (\ref{ode})  
\begin{equation}A_{n+1}\begin{bmatrix}x-\alpha_{n}& -\beta_{n}\\ 1&0\end{bmatrix}=\begin{bmatrix}x-\alpha_{n}& -\beta_{n}\\ 1&0\end{bmatrix}A_n+\begin{bmatrix} 1& 0\\ 0&0\end{bmatrix},\end{equation}
\noindent where the second matrix has determinant $\beta_{n}>0$, hence the  $A_n$ are uniquely determined. We can therefore follow the approach of section VI of [32]. From the differential equation (\ref{ode}), 
\begin{align} \Bigl( {\frac{\partial}{\partial x}}+{\frac{\partial}{\partial y}}\Bigr) Q_n(x,y)&= {\frac{1}{h_{n-1}}}\Bigl( {\frac{\partial}{\partial x}}+{\frac{\partial}{\partial y}}\Bigr) {\frac{\langle JY_n(x), Y_n(y)\rangle}{x-y}}\nonumber\\
&={\frac{1}{h_{n-1}}}{\frac{\langle (JA_n(x)Y_n(x), Y_n(y)\rangle +\langle JY_n(x), A_n(y)Y_n(y)\rangle}{x-y}}\nonumber\\
&={\frac{1}{h_{n-1}}} \langle B_n(x,y)Y_n(x), Y_n(y)\rangle\end{align}
where $B_n(x,y)=JA_n(x)+A_n(y)^TJ$ is given explicitly by
\begin{equation}B_n(x,y)=\begin{bmatrix} {\frac{ (a_n\Theta_{n-1}/W)(y)-(a_n\Theta_{n-1}/W)(x)}{x-y}}& {\frac{(\Omega_n/W)(x)-(\Omega_n/W)(y)}{x-y}}\\
{\frac{(\Omega_n/W)(x)-(\Omega_n/W)(y)}{x-y}}& {\frac{(a_n\Theta_n/W)(y)-(a_n\Theta_n/W)(x)}{x-y}}\end{bmatrix},\end{equation}
which is rational, symmetric with respect to interchange of variables $x\leftrightarrow y$ and symmetric with respect to matrix transpose. From the identity $Ww'_0=2w_0V$, and cancelling any common zeros of $V$ and $W$, we deduce that $W$ has no zeros on $(0, \infty )$ since $w_0(x)>0$ for all $x>0$ by hypothesis. Observe also that $\int_0^\infty x^kw_0(x)dx$ is finite for all $k\in {\mathbb N}\cup \{ 0\}$. By selecting the products of functions that depend on one variable, namely $x$ or $y$, we can therefore choose $\phi_j$ and $\phi_k$ from among the functions in $B$ and $Y$ such that $\phi_j, \psi_j\in L^2(0,\infty )$ and 
\begin{equation}\Bigl( {\frac{\partial}{\partial x}}+{\frac{\partial}{\partial y}}\Bigr) Q_n(x,y)=-\sum_{j=1}^N \phi_j(x)\psi_j(y).\end{equation}
By integration, we obtain 
\begin{equation} Q_n(x,y)=\sum_{j=1}^N \int_0^\infty \phi_j(x+t)\psi_j(t+y)\, dt +q(x-y)\end{equation}
where $q(x-y)\rightarrow 0$ as $x\rightarrow\infty$ or $y\rightarrow\infty$, so $q=0$. We can select the $\phi_j, \psi_j$ so that $\int_0^\infty x\vert \phi_j(x)\vert^2\, dx$ and $\int_0^\infty x\vert\psi_j(x)\vert^2\, dx$ are all finite, so $\Gamma_{\phi_j}$ and $\Gamma_{\psi_j}$ are Hilbert--Schmidt.\par
\indent (ii) Let $h(x)=1-e^{-f(x)}$ for some $f\in L^\infty$, so that $e^{-\sum_{j=1}^n f(x_j)}=\prod_{j=1}^n (1-h(x_j)).$ Then 
\begin{align}{\mathbb E} e^{-\sum f}&={\frac{ \int_{(0,\infty )^n} \prod_{1\leq j<k\leq n} (x_j-x_k)^2\prod_{j=1}^n (1-h(x_j))w_0(x_j)dx_j}{
\int_{(0,\infty)^n} \prod_{1\leq j<k\leq n} (x_j-x_k)^2\prod_{j=1}^n w_0(x_j)dx_j}}\nonumber\\
&=\det (I-M_hQ_n).\end{align}
We let $\Phi :(0, \infty)\rightarrow {\mathbb C}^{N\times 1}$ be $\Phi (x)={\hbox{column}}[\phi_j(x)]_{j=1}^N$ and $\Psi :(0, \infty)\rightarrow {\mathbb C}^{N\times 1}$ be $\Psi (x)={\hbox{column}}[\psi_j(x)]_{j=1}^N$, 
as in the Corollary, so
\begin{align}{\mathbb E} e^{-f}&=\det (I-M_h\Gamma_{\Phi}^T\Gamma_{\Psi})\nonumber\\
&=\det (I-\Gamma_{\Psi}M_h\Gamma_{\Phi}^T)\end{align}
where the final operator has a matrix kernel
\begin{equation}\Gamma_{\Psi}M_h\Gamma_{\Phi}^T\leftrightarrow \Bigl[\int_0^\infty \psi_j(x+u)(1-e^{-f(u)})\phi_k(u+y)\, du\Bigr]_{j,k=1}^N.\end{equation}
\indent (iii)
For $\Re \beta >0$, the point $\lambda =1-e^{-\beta}$ lies in the disc of centre $1$ and radius $1$ in ${\mathbb C}$. Then for the step function $f(x)=\beta \vartheta (x-t)$ we have
\begin{align}{\mathbb E}e^{-\beta \sum \vartheta (.-t)}&=\sum_{k=0}^\infty e^{-k\beta}{\mathbb {P}}[\natural \{ j: x_j\in (t,\infty )\} =k]\nonumber\\
&=\sum_{k=0}^\infty (1-\lambda )^k{\mathbb {P}}[\natural \{ j: x_j\in (t,\infty )\} =k]\nonumber\\
&=\sum_{k=0}^\infty {\frac{(1-\lambda )^k}{k!}}\Bigl({\frac{d^k}{d\mu^k}}\Bigr)_{\mu =1} \det (I-\mu M_{\vartheta (.-t)}Q_n\bigr),\end{align} 
so we have the moment generating function of the number of the $x_j$ that are greater than $t$. Then

\begin{equation}\label{detstep}(1-e^{-\beta})\Gamma_{\Psi}M_{\vartheta (. -t)}\Gamma_{\Phi}^T
\leftrightarrow \lambda\Bigl[\int_0^\infty \psi_j(x+u+t)\phi_k(u+t+y)\, du\Bigr]_{j,k=1}^N,\end{equation}
where each entry of the matrix is a product of Hankel operators, with scattering functions $\psi_j(x), \phi_k(x)$ shifted to $\psi_j(x+t), \phi_k(x+t)$. \par 
\end{proof}

\indent  Theorem \ref{Hankelfactors} involves a Fredholm determinant. The following result gives an equivalent expression involving finite determinants on the numerator. We introduce the block matrix
\begin{equation}\Theta_t=\begin{bmatrix} 0_{N\times N}& \Psi_t&0_{N\times (N-1)}\\
                              \Phi_t^T&0&0_{1\times (N-1)}\\
                              0_{(N-1)\times N}& 0_{(N-1)\times 1}& 0_{(N-1)\times (N-1)}\end{bmatrix}.\end{equation}

\begin{cor}\label{finitedet}{Suppose that $\Gamma_{\Theta_t}\in {\mathcal L}^1$ and $I+\sqrt{\lambda} \Gamma_{\Theta_t}$ is invertible.\par
(i) Then for any orthogonal projection $P_n$ on $L^2((0, \infty ); {\mathbb C}^{2N})$ with $P_n^\perp =I-P_n$, 
\begin{equation}{\mathbb E}e^{-\beta \sum \vartheta (.-t )}={\frac{\det P_n(I+\sqrt{\lambda} \Gamma_{\Theta_t})P_n}{\det P^\perp_n(I+\sqrt{\lambda} \Gamma_{\Theta_t})^{-1}P^\perp_n}}.
\end{equation}
(ii) Let $L_j(x)$ be the Laguerre polynomial, and let $P_n$ be the orthogonal projection onto 
$${\hbox{span}}\{ e^{-x/2}L_j(x): j=0, \dots, n\} \otimes {\mathbb C}^{2N}.$$
Then $P_n\Gamma_{\Theta_t}P_n$ is unitarily equivalent to a finite block Hankel matrix.}\end{cor}                             
  
\begin{proof} (i) We have especially chosen $\Theta$ so that by Theorem \ref{Hankelfactors}(iii), we have
\begin{equation}{\mathbb E}e^{-\beta \sum \vartheta (.-t )}=\det (I+\sqrt{\lambda} \Gamma_{\Theta_t}).\end{equation}  
Then the stated result follows from a determinant formula credited to Jacobi; see [2].\par   
(ii) Hankel integral operators correspond to Hankel matrices via the Laguerre orthonormal basis of $L^2(0, \infty );$ see [24], page 53. (This is a special feature of the Laguerre polynomials.) To extend this to Hankel integral operators on $L^2((0, \infty ); {\mathbb C}^{2N})$, we just compute the block Hankel matrix
$$\Bigl[\int_0^\infty \Theta_t(x)e^{-x/2}L_{j+k}(x)\, dx\Bigr]_{j,k=0}^{n-1}$$
which has $(2N)\times (2N)$ block entries, and the cross-diagonal pattern that is characteristic of Hankel matrices.    
\end{proof}

\indent Theorem \ref{Hankelfactors} shows that replacing $w(x)$ by $w(x)e^{-\beta\vartheta (x-t)}$ corresponds shifting $\Theta_0$ to $\Theta_t$. The shift operation is simple to describe in terms of linear systems, as in (\ref{linsys2}). Unfortunately, $\vartheta$ is discontinuous, so $w(x)e^{-\beta\vartheta (x-t)}$ is not itself a semiclassical weight, and we cannot immediately deduce a differential equation such as \ref{ode} for orthogonal polynomials generated by $w(x)e^{-\beta\vartheta (x-t)}$. Nevetheless, Min Chao and Chen [21] derived an ODE for gap probabilities in the Jacobi ensemble. \par  
\indent Hence we replace the step function by 
\begin{equation}f(x)=\beta\tan^{-1} {\frac{x-t}{\varepsilon}}\end{equation}
for $\varepsilon>0,$ and $t, \beta\in {\mathbb R}$, since $f(x)\rightarrow\beta \pi (\vartheta (x-t)-1/2)$ as $\varepsilon\rightarrow 0+$.\par
\indent As in Theorem \ref{Hankelfactors}, we suppose that $w_0$ satisfies $Ww_0'=2Vw_0$, where $V,W$ are polynomials, and let $v_0=-\log w_0$. 
Then there exists $\varepsilon_0>0$ such that 
$$(2V(x)(x-z_+)(x-z_-)+i\beta (z_+-z_-)W(x), W(x)(x-z_+)(x-z_-))$$ 
is also generic for all real $\beta$ and $0<\Im z_+<\varepsilon_0$ and $0<-\Im z_-<\varepsilon_0$. In particular, we can replace our previous weight $w_0(x)$ by
\begin{equation}w(x)=w_0(z) (x-z_+)^{i\beta/2}(x-z_-)^{-i\beta /2}\end{equation}
then we build the system of monic orthogonal polynomials $(p_j(x))_{j=0}^\infty$ for the complex bilinear form $\langle f,g\rangle=\int_E f(x)g(x)w(x)\, dx$.
\begin{prop}\label{Pain} Suppose that $(2V,W)$ is generic.\par
(i) There exists $\varepsilon_0>0$ such that
\begin{equation}\label{parameters} \bigl(2V(x)(\varepsilon^2+(x-t)^2)+\varepsilon\beta W(x), W(x)(\varepsilon^2+(x-t)^2)\bigr),\end{equation}
is generic for all real $\beta$ and $0<\varepsilon<\varepsilon_0$; 
\par
(ii) there exists a consistent system of ordinary differential equations as in (\ref{iso}) 
\begin{align}\label{mainode}{\frac{dY}{dx}}& =A(x,t; \beta , \varepsilon )Y\\
{\frac{dY}{dt}}& =H(x,t; \beta , \varepsilon )Y,
\end{align}
where $A(x,t; \beta , \varepsilon )$ is a proper rational function of $x$ with trace zero, and simple poles at the zeros of $W$ and $t\mp i\varepsilon$;\par
(iii) the consistency condition holds
\begin{equation}\label{consist} {\frac{\partial A}{\partial t}}-{\frac{\partial H}{\partial x}}+\bigl[ A, H \bigr]=0.\end{equation}\end{prop} 

\begin{proof} (i) This is a direct check of the definitions. Then the modified potential $v=-\log w$ has $v'$ rational, and we obtain a family of pairs of polynomials, depending upon parameters $(t, \varepsilon, \beta )$. For given $n$, we can choose $\varepsilon_0>0$ such that the Gram-Schmidt process for the bilinear form $\langle f,g\rangle$ produces orthogonal polynomials of degree up to $n$, for all $0<\varepsilon<\varepsilon_0$.\par
(ii) Magnus [19] obtains $\Theta_n$ and $\Omega_n$ by recursion, and one checks that the degree of $\Theta_n$ is less than or equal to $m$, while the degree of the denominator is $m+2$. From his recursion formula (20), the degree of $\Omega_n^2$ is less than or equal to $2(m+1)$, so $A(x,t; \beta , \varepsilon )$ is strictly proper. 
By Proposition \ref{Pain}(ii), we can write
\begin{equation}\label{A(x)} A(x,z_{\pm})={\frac{A_+}{x-z_+ }}+{\frac{A_-}{x-z_-}}+\sum_{j=1}^m {\frac{A_j}{x-\alpha_j}}\end{equation}
where the $2\times 2$  residue matrices $A_j, A_{\pm}$ depend upon $(\beta,z_{\pm} )$, but not upon $x$. The set of singular points in the Riemann sphere ${\mathbb C}\cup\{ \infty\}$ is $\{ \alpha_1, \dots, \alpha_m, z_{\pm} , \infty\}$. \par
We can take $z_{\pm}=t\pm i\varepsilon$, a complex conjugate pair. Then we fix $\beta\in {\mathbb R}$ and some $0<\varepsilon <\varepsilon_0$ and regard $t$ as the main deformation parameter. Then the weight
\begin{equation}\label{newweight}w(x)=w_0(x)\exp\Bigl( {\frac{\beta \pi}{2}}-\beta\tan^{-1}{\frac{x-t}{\varepsilon}}\Bigr)\end{equation}
is positive and continuous on $E$, so $p_j$ is a real polynomial and $h_j>0$. 
Since the differential equation (\ref{mainode}) has only regular singular points, the monodromy is fully described in [25] by results of Schlesinger page 148 and Dekkers page 180 in terms of connections of dimension two on the punctured Riemann sphere.
Schlesinger found the condition for the system to undergo an infinitesimal change in the poles $\{ \alpha_1, \dots, \alpha_m; z_{\pm}\}$ that does not change the monodromy. Let $Y$ be the fundamental solution matrix of (\ref{mainode}), and introduce \begin{align}H&={\frac{\partial Y}{\partial t}}Y^{-1}\\
 &=\Bigl({\frac{\partial Y}{\partial z_+}}+{\frac{\partial Y}{\partial z_-}}\Bigr)Y^{-1}\\
 &=-{\frac{A_+}{x-z_+}}-{\frac{A_-}{x-z_-}}\end{align}
  to obtain the required variation in $z_{\mp}$.\par
\indent (iii)  This formula follows from the equality of mixed partial derivatives $\partial^2Y/\partial t\partial x=\partial^2Y/\partial x\partial t$ where $Y$ is the fundamental solution matrix of (\ref{mainode}) and
$\partial /\partial t=\partial /\partial z_++\partial /\partial z_-$. To ensure that the differential equations are indeed consistent, we require
\begin{equation} {\frac{\partial A(x,z_{\pm})}{\partial t}}={\frac{A_+}{(x-z_+)^2 }}+{\frac{A_-}{(x-z_-)^2}}+{\frac{{\frac{\partial A_+}{\partial t}}}{x-z_+}}+{\frac{{\frac{\partial A_-}{\partial t}}}{x-z_-}}+\sum_{j=1}^m {\frac{{\frac{\partial A_j}{\partial t}}}{x-\alpha_j}},\end{equation}
where by Schlesinger's equations
\begin{equation} 
{\frac{\partial A_{\pm}}{\partial t}}=-\sum_{j=1}^m {\frac{[A_j, A_{\pm}]}{\alpha_j-z_{\pm}}},\end{equation}
\begin{equation} {\frac{\partial A_j}{\partial t}}={\frac{[A_j, A_+]}{\alpha_j-z_+}}+{\frac{[A_j, A_-]}{\alpha_j-z_-}}\qquad (j=1, \dots, m).\end{equation}
\end{proof} 

\begin{cor}\label{PVI}Suppose in (\ref{A(x)}) that $m=1$, that $A_++A_-+A_1$ is a diagonal matrix and 
\begin{equation} {\hbox{trace}}\, A_+={\hbox{trace}}\, A_-={\hbox{trace}}\, A_1=0.\end{equation}
Then (\ref{consist}) reduces to a Painlev\'e VI equation.\end{cor}
\begin{proof} By translating $z$ to $z+t$, we replace the singular points $(t-i\varepsilon , t+i\varepsilon, \alpha_1, \infty )$  by $(-i\varepsilon , +i\varepsilon, \alpha_1-t, \infty )$, so we have variation in only one pole. Then we can apply known results from [15] and [17] to reduce the compatibility condition (\ref{consist}) to a Painlev\'e VI ordinary differential equation. 
\end{proof}

\begin{rem} (i) Chen and Its [7] showed that the Hankel determinant $D[w]$ gives the isomonodromic $\tau$ function for the system of Schlesinger equations that describe the isomonodromic deformation of (\ref{mainode}) with respect to the position of the poles. The Schlesinger equations may be solved in terms of the $\Theta$-function on a hyperelliptic Riemann surface, as in [18]. The solutions to the monodromy preserving differential equations have singularities which are poles, except for the fixed singularities.
Previously, Magnus [19] had found conditions for the system (\ref{iso}) to undergo an isomonodromic deformation, and obtained examples that realise the nonlinear Painlev\'e VI equation as (\ref{consist}).\par
(ii) Tracy and Widom considered Fredholm determinants $\det (I-\Gamma_{\phi_t}^2)$ for classical orthogonal polynomials [31, 32] and computed $(d/dt)\log\det (I-\Gamma_{\Phi_t}\Gamma_{\Psi_t})$ in terms of operator kernels. They identified weights that produce Painlev\'e $II$, $III$, $IV$ and $V$. For differential equations (\ref{ode}) with $W(x)=1$, that have polynomial coefficients,  Palmer [22] identified $\det (I-\Gamma_{\phi_t}^2)$ as the $\tau$-function of the ODE (\ref{ode}) for isomonodromic deformations. His analysis addressed the case in which infinity is an irregular singular point. \par  
(iii) By taking $\varepsilon \rightarrow 0+$, have $z_{\pm}\rightarrow t$ and 
\begin{equation}\label{Dtheta}D_n[w]= 
\det\Bigl[\int_{0}^\infty x^{j+k}e^{2\pi \beta (1-\vartheta (x-t))}w_0(x)\, dx\Bigr]_{j,k=0}^{n-1}.\end{equation}
In section 7, we consider the behaviour of this determinant for large $n$.
\end{rem}
\end{section}
\begin{section} {Wiener--Hopf Factorization} 
\vskip.05in
\noindent Fix $0<\varepsilon <1$. Let ${\mathcal C}_2^0={\mathcal C}_2^0(\varepsilon )$ be the space of functions $f$ such that:\par
\indent (i) $f$ is bounded and analytic on $\{ z: \vert \Re z\vert <\varepsilon\}$;\par
\indent (ii) $f(\eta +i\xi )\rightarrow 0$ as $\xi\rightarrow\pm \infty$, uniformly for $\vert \eta\vert \leq \varepsilon /2$;\par
\indent (iii)
$$\sup_{\vert \eta \vert <\varepsilon /2} \int_{-\infty}^\infty \vert f(\eta+i\xi  )\vert^2 d\xi <\infty.$$
\noindent  Let ${\mathcal C}_2={\mathcal C}_2^0+ {\mathbb C}$.\par
\vskip.05in
\begin{prop}\label{Hom}(i) Then ${\mathcal C}_2$ is a commutative and unital Banach algebra under the usual pointwise multiplication,\par
\indent  (ii) there is a bounded linear map $\psi \mapsto \Gamma_\phi^\dagger $ from ${\mathcal C}^0_2\rightarrow {\mathcal L}^2$ via the transform (\ref{transform}).\end{prop}
\vskip.05in
\begin{proof}(i) We take the norm to be
\begin{equation}\Vert f\Vert_{{\mathcal C}_2}=\sup_z\Bigl\{ \vert f(z)\vert :z=x+iy; y\in {\mathbb R},x\in  (-\varepsilon ,\varepsilon )\Bigr\} +\sup_{-\varepsilon/2<x<\varepsilon /2}\Bigl( \int_{-\infty}^\infty \vert f(x+iy)\vert^2 \, dy\Bigr)^{1/2}.\end{equation}
\noindent  Evidently ${\mathcal C}_2$ is a subspace of the Banach algebra $H^\infty$ of bounded functions on the strip $\{ z: \vert \Im z\vert <\varepsilon \}$, hence ${\mathcal C}_2$ is an integral domain. \par

\indent (ii) The Hankel integral operator with kernel $\phi (s+t)$ on $L^2(0, \infty )$ has Hilbert--Schmidt norm satisfing 
\begin{align}\Vert \Gamma (\phi )\Vert_{{\mathcal L}^2}^2&=\int_0^\infty t\vert\phi (t)\vert^2\, dt\nonumber\\
&\leq \int_0^\infty (1+t^2)\vert\phi (t)\vert^2\, dt\nonumber\\
&=\int_{-\infty}^\infty \bigl\vert  \psi (i\xi )\bigr\vert^2 \, d\xi+\int_{-\infty}^\infty \bigl\vert \psi '(i\xi )\bigr\vert^2 \, d\xi,\end{align}
\noindent where we have used Plancherel's formula. By Cauchy's integral formula for derivatives, we have
\begin{align}\int_{-\infty}^\infty \bigl\vert \psi '(i\xi )\bigr\vert^2 \, d\xi&\leq {\frac{1}{\pi\varepsilon}}\int_0^{2\pi}\int_{-\infty}^\infty \bigl\vert \psi (i\xi +\varepsilon e^{i\theta}/2)\bigr\vert^2 \, d\xi d\theta\nonumber\\
&\leq {\frac{2}{\varepsilon}}\sup_{\vert \eta \vert <\varepsilon /2} \int_{-\infty}^\infty \vert \psi (i\xi +\eta )\vert^2 d\xi.\end{align}
\noindent  Hence $\Gamma (\check f)$ is a Hilbert--Schmidt operator.\par
\end{proof}
\indent Employing more classical language, Titchmarsh [30] identified a subgroup of $G({\mathcal C}_2)/\exp ({\mathcal C}_2)$ with ${\mathbb Z}$.  Let $\psi$ be typical element of ${\mathcal C}_2$ such that $\psi (z)\rightarrow 1$ as $z\rightarrow \pm \infty$ along the imaginary axis and such that $\psi$ has no zeros on the imaginary axis. The function $\vert \log\psi (\eta+i\xi )\vert$ is square integrable for $-\varepsilon/2 \leq \eta\leq \varepsilon/2$. Then $\psi$  has the form
\begin{equation} \psi (z) =\Bigl({\frac{z-1}{z+1}}\Bigr)^k {\frac{\prod _{j=1}^m (z-w_j)} {(z^2-1)^{m/2}}} \exp \bigl(\chi_+(z)-\chi_-(z)\bigr)\end{equation}
\noindent where (1) $w_j$ are the zeros of $\psi (z) $ for $\vert \Re z\vert \leq \varepsilon /2$,\par
\indent (2)  $k$ is the winding number of the contour $\{ \psi (i\xi ): -\infty \leq \xi \leq \infty \}$,\par
\indent (3) $\chi_+$ is holomorphic and bounded on $\Re z\geq -\varepsilon/2$ and\par
\begin{equation}\chi_+(w)={\frac{1}{2\pi i}}\int_{-i\infty -\varepsilon /2}^{i\infty -\varepsilon /2}{\frac{\log \psi (z)}{z-w}}dz,\end{equation}
\indent (4) $\chi_-$ is holomorphic and bounded on  $\Re z\leq \varepsilon/2$ with \par
\begin{equation}\chi_-(w)={\frac{1}{2\pi i}}\int_{-i\infty +\varepsilon /2}^{i\infty +\varepsilon /2}{\frac{\log \psi (z)}{z-w}}dz.\end{equation}
See also the results of Rappaport from [27].\par

\indent The spaces $H^\infty (\{ s: \Re s<\varepsilon \})$ and $H^\infty (\{ s: -\varepsilon <\Re s\})$ have intersection ${\mathbb C}$ by Liouville's theorem, so $\chi_+$ and $\chi_-$ are unique up this additive constant. If $\psi\in G({\mathcal C}_2)$, then $\psi$ has no zeros and the middle factor is absent, but we are left with the initial factor incorporating the winding number.\par
\vskip.05in
 Let $Q$ be an orthogonal projection on  $L^2((0, \infty ); dx)$, and introduce the complementary spaces $H_+ ={\mathcal T}QL^2((0, \infty ); dx)$ and $H_- ={\mathcal T}(I-Q)L^2((0, \infty ); dx).$  \par
\indent For $g\in L^\infty$, let $M_g\in {\mathcal L}(L^2)$ be the multiplication operator $M_g:h\mapsto gh$. Then we introduce $W_g\in {\mathcal L}^\infty (H_+)$,  $\Gamma_g\in {\mathcal L}^\infty (H_+, H_-),$ , 
 $\tilde W_g\in {\mathcal L}^\infty (H_-)$,  $\tilde \Gamma_g\in {\mathcal L}^\infty (H_-, H_+),$ by
\begin{equation}M_g=\begin{bmatrix}W_g&\tilde \Gamma_g \\ \Gamma_g&\tilde W_g\end{bmatrix}\quad
\begin{matrix} H_+\\ H_- \end{matrix} .\end{equation}
\vskip.05in
\begin{lem}Let ${\mathcal C}_p$ be the space of $g\in L^\infty$ such that $\Gamma_g\in {\mathcal L}^p$ and $\tilde \Gamma_g\in{ \mathcal L}^p$, and let 
\begin{equation}\Vert g\Vert_{{\mathcal C}_p}=\max \{ \Vert W_g\Vert_{{\mathcal L}^\infty},  \Vert \tilde W_g\Vert_{{\mathcal L}^\infty}\} +\Vert \Gamma_g\Vert_{{\mathcal L}^p}+\Vert \tilde\Gamma_g\Vert_{{\mathcal L}^p}.\end{equation}
\noindent Then ${\mathcal C}_p$ is a subalgebra of ${L}^\infty$ such that
\begin{equation}\Vert gh\Vert_{{\mathcal C}_p}\leq \Vert g\Vert_{{\mathcal C}_p}\Vert h\Vert_{{\mathcal C}_p}.\end{equation}\end{lem}
\vskip.05in
\begin{proof} For $g\in L^\infty$ we have $M_g\in {\mathcal L}^\infty$, and $\Vert g\Vert_{L^\infty}\leq \Vert M_g\Vert_{{\mathcal L}^\infty}\leq \Vert g\Vert_{{\mathcal C}_p}$, so the pointwise multiplication is unambiguously defined. 
Conversely, suppose that $g,h\in {\mathcal C}_p$, and observe that 
\begin{equation}\begin{bmatrix}W_{gh}&\tilde \Gamma_{gh} \\ \Gamma_{gh}&\tilde W_{gh}\end{bmatrix}=\begin{bmatrix}W_g&\tilde \Gamma_g \\ \Gamma_g&\tilde W_g\end{bmatrix}\begin{bmatrix}W_h&\tilde \Gamma_h \\ \Gamma_h&\tilde W_h\end{bmatrix}=\begin{bmatrix}W_gW_h+\tilde\Gamma_g\Gamma_h&W_g\tilde \Gamma_h +\tilde \Gamma_g \tilde W_h\\ \Gamma_g W_h+\tilde W_g\Gamma_h&\tilde W_g\tilde W_h+\Gamma_g\tilde \Gamma_h\end{bmatrix}\end{equation}
\noindent leading to identities such as 
\begin{equation} W_{gh}=W_gW_h+\tilde \Gamma_g\Gamma_h\end{equation}
\begin{equation}\Gamma_{gh}=\Gamma_gW_h+\tilde W_g\Gamma_h.\end{equation}
\noindent The ideal property of the Schatten norm gives 
\begin{equation}\Vert W_{gh}\Vert_{{\mathcal L}^\infty}\leq \Vert W_g\Vert_{{\mathcal L}^\infty}  \Vert W_h\Vert_{{\mathcal L}^\infty} + \Vert \tilde\Gamma_g\Vert_{{\mathcal L}^p} \Vert\Gamma_h\Vert_{{\mathcal L}^p},\end{equation}
\begin{equation} \Vert \Gamma_{gh}\Vert_{{\mathcal L}^\infty}\leq \Vert \Gamma_g\Vert_{{\mathcal L}^p}  \Vert W_h\Vert_{{\mathcal L}^\infty} + \Vert \tilde W_g\Vert_{{\mathcal L}^\infty} \Vert\Gamma_h\Vert_{{\mathcal L}^p},\end{equation}
\noindent and similar inequalities for each entry of (), hence the norm satisfies the submultiplicative property.\par
\end{proof}
\vskip.05in
\indent Let ${\mathcal A}_2$ be the subalgebra of ${\mathcal C}_2$ consisting of $f\in {\mathcal C}_2$ such that $f$ is bounded and holomorphic on the right half plane, and let ${\mathcal A}^*_2$ be the subalgebra of ${\mathcal C}_2$ consisting of $f\in {\mathcal C}_2$  such that $f(z)=\bar g(-\bar z)$ for some $g\in {\mathcal A}_2.$ Note that ${\mathcal A}^*_2\cap {\mathcal A}_2={\mathbb C}1$ by Liouville's theorem. The following result describes $\psi\in G({\mathcal A}^*_2)G({\mathcal A}_2)$ that has no imaginary zeros, but may have zeros elsewhere. For ${\mathcal G}$ a group, we write $\{X,Y\}=XYX^{-1}Y^{-1}$ for the multiplicative commutator. \par
\vskip.05in
\begin{lem}\label{Wienerhopf} (1)Suppose that $\psi\in {\mathcal C}_2$ has no zeros on the imaginary axis,\par
(2)  $\psi (i\xi +\eta )\rightarrow 1$ as $\xi\rightarrow\pm\infty$, uniformly for $-\varepsilon <\eta <\varepsilon$,\par
(3) the winding number of $\{ \psi (i\xi ):-\infty \leq \xi\leq \infty \}$ is zero, and\par
(4) $\psi (z)=1+O(1/\vert z\vert^{1/2+\delta })$ as $\vert z\vert\rightarrow\infty$ for some $0<\delta\leq 1/2$.\par  
\noindent Then there exists $0<\varepsilon'<\varepsilon$ such that  $\psi$ has a Wiener--Hopf factorization 
\begin{equation}\label{wienerhopf}\psi =\psi_-\psi_+\qquad (-\varepsilon'<\Re z<\varepsilon')\end{equation}  such that \par
\indent (i) $\psi_{-}$ is bounded, holomorphic and free from zeros on $\{ z: \Re z<\varepsilon'/2\}$\par
\indent (ii) $\psi_+$ is bounded, holomorphic and free from zeros on $\{ z: \Re z>-\varepsilon'/2 \}$;\par
\indent (iii) $\psi_{\pm} (\eta +i\xi )=1+O(1/\vert z\vert^{(1+\delta )/2})$ as $\xi\rightarrow\pm\infty, $ uniformly for $-\varepsilon'/\varepsilon <\eta<\varepsilon'/\varepsilon$.\end{lem}
\begin{proof} (i), (ii) By hypothesis, $\psi$ has no zeros lie on the imaginary axis, and only finitely many in the strip $\{ z: -\varepsilon <\Re z<\varepsilon\}$; so by choosing $0<\varepsilon'<\varepsilon$ sufficiently small, we can ensure that $\psi $ is free from zeros $\{ z: -\varepsilon \leq\Re z\leq \varepsilon\}$. Then we choose
\begin{equation}\chi_-(z)=\int_{-i\infty +\varepsilon' }^{i\infty +\varepsilon}{\frac{\log \psi (z)}{z-w}}dz,\end{equation}
\begin{equation}\chi_+(z)={\frac{1}{2\pi i}}\int_{-i\infty -\varepsilon' }^{i\infty -\varepsilon' }{\frac{\log \psi (z)}{z-w}}dz;\end{equation}
\noindent then the functions $\psi_-(z)=\exp (-\chi_-(z))$ and $\psi_+(z)=\exp (\chi_+(z))$  satisfy $\psi_-\psi_+=\psi$, as in (\ref{wienerhopf}). Also, we can introduce $R>0$ such that 
\begin{equation}\sup\{ \vert \psi (\eta +i\xi )-1\vert : -\varepsilon' <\eta <\varepsilon', \xi \in (-\infty ,-R)\cup (R, \infty )\}<1/2\end{equation}
\noindent and $\psi (z)$ is free from zeros on $\{ z=\eta+i\xi : -\varepsilon' <\eta <\varepsilon', \xi \in [-R,R]\}$. Then one can introduce $M$ such that
\begin{equation}\bigl\vert\log \psi (z)\bigr\vert\leq {\frac{M}{(1+\vert\xi\vert )^{1/2+\delta}}}\qquad (z=\eta +i\xi ; -\infty <\xi<\infty, \quad -\varepsilon' <\eta <\varepsilon' ).\end{equation}
\noindent The convolution of a pair of $L^2$ functions gives a continuous function which vanishes at infinity, so $\chi_\pm$ are bounded and holomorphic on the smaller half planes determined by abscissae $\pm\varepsilon'/2$.\par
\indent (iii) To obtain the more precise estimate of (iii), we consider $z$ with $-\varepsilon'/2\leq \Re z\leq \varepsilon'/2$ and $\Im z>0$ large; then we take $\xi_0=\vert z\vert/2$ and $p>2/\delta$ with conjugate $q=p/(p-1)$ and split the integral
\begin{align}\vert\chi_-(z)\vert&\leq \int_{-\infty}^\infty {\frac{\vert\log \psi (\varepsilon'+i\xi )\vert}{\vert z-\varepsilon'-i\xi\vert }}d\xi\nonumber\\
&=\int_{-\infty}^{\xi_0}+\int_{\xi_0}^\infty  {\frac{\vert\log \psi (\varepsilon'+i\xi )\vert}{\vert z-\varepsilon'-i\xi\vert }}d\xi\cr
&\leq \Bigl(\int_{-\infty}^{\xi_0}{\frac{M^q}{(1+\vert\xi\vert )^{q/2+q\delta}}}d\xi\Bigr)^{1/q}\Bigl( \int_{-\infty}^{\xi_0}{\frac{d\xi}{\vert z-\varepsilon'-i\xi\vert^p}}\Bigr)^{1/p}\nonumber\\
&\quad +
\Bigl(\int_{\xi_0}^\infty {\frac{M^p}{(1+\vert\xi\vert )^{p/2+p\delta}}}d\xi\Bigr)^{1/p}\Bigl( \int_{\xi_0}^\infty{\frac{d\xi}{\vert z-\varepsilon'-i\xi\vert^q}}\Bigr)^{1/q}\nonumber\\
&=O(1/\vert z\vert^{1-1/p})+O(1/\vert z\vert^{1/2+\delta -1/p}),\end{align}
\noindent where we have used H\"older's inequality on the integrals. The other estimates in (iii) are similar. Likewise, one can show that $\chi'_{\pm} (z)=O(1/\vert z\vert^{(1+\delta )/2})$ as $\xi\rightarrow\pm\infty).$\par
\end{proof}
\vskip.05in 
\end{section} 
\begin{section}{Wiener--Hopf determinant}\par
\vskip.05in
\noindent This section contains the main theoretical result, as follows.\par
 
\begin{thm}\label{WHfactors}Suppose that $\psi\in L^\infty (i{\mathbb R})$ has a Wiener--Hopf factorization $\psi=\psi_-\psi_+$ as in Lemma \ref{Wienerhopf}. Then there exists a $2\times 2$ scattering function
\begin{equation}\Phi (x)= \begin{bmatrix}0&\phi_1(x)\\ \phi_2(x)& 0\end{bmatrix}\end{equation}
\noindent such that  Hankel operators integral operators $\Gamma_{\phi_1}$ and $\Gamma_{\phi_2}$ are Hilbert--Schmidt on $L^2(0, \infty )$ and
\begin{equation}1/\det (W(\psi )W(\psi^{-1} )) =\det (I-\Gamma_{\phi_1}\Gamma_{\phi_2})=\det_2 (I+\Gamma_{\Phi}).\end{equation} 
There are three particular cases that arise under the following hypotheses:\par 
\indent (i) $\psi_-(i\xi )/\bar \psi_- (-i\xi ) =\psi_+(i\xi )/\bar \psi_+ (-i\xi )$ if and only if $\phi_1$ and $\phi_2$ are real, so that $\Gamma_{\phi_1}$ and $\Gamma_{\phi_2}$ are self-adjoint; \par
(ii) $\psi_-(i\xi )\psi_-(-i\xi )=\psi_+(i\xi )\psi_+(-i\xi ),$ if and only if $\phi_1=\phi_2$, in which case the Carleman determinants satisfy\par
\begin{equation}\det_2 (I+\lambda\Gamma_{\Phi}) 
=\det_2(I-\lambda\Gamma_{\phi_1})\det_2(I+\lambda\Gamma_{\phi_1})\qquad (\lambda \in {\mathbb C});\end{equation}
(iii) $\vert\psi_-(i\xi )\vert =\vert\psi_+(i\xi )\vert$ if and only if the operator $\Gamma_\Phi$ is self-adjoint.\par
\noindent Any pair of these conditions implies the other one.\end{thm}
\vskip.1in
\begin{proof} By the Lemma \ref{Wienerhopf}, we can choose $\varepsilon'>0$ such that
\begin{equation}f={\frac{\psi_-}{\psi_+}}-1,\qquad g= {\frac{\psi_+}{\psi_-}}-1\end{equation}
\noindent both belong to ${\mathcal C}_2={\mathcal C}_2(\varepsilon')$ and satisfy $fg=2-\psi_-/\psi_+-\psi_+/\psi_-$; hence 
\begin{align}\tilde \Gamma (f)\Gamma (g)&=W(fg)-W(f)W(g)\nonumber\\
&=I-W(\psi_-/\psi_+)W(\psi_+/\psi_-).\end{align}
\noindent Now  the operators $W(\psi_{\pm})$ are invertible, and $W(\rho\psi_+)=W(\rho )W(\psi_+)$ and $W(\psi_-\rho )=W(\psi_-)W(\rho)$ for all $\rho\in {\mathcal C}_2$. So we have
\begin{align}W(\psi_-/\psi_+)W(\psi_+/\psi_-)&=W(\psi_-)W(\psi_+)^{-1}W(\psi_+/\psi_-)\nonumber\\
&=W(\psi_-)W(\psi_+)^{-1}W(\psi_-)^{-1}W(\psi_+)\nonumber\\
&=\{ W(\psi_-), W(\psi_+)^{-1}\}.\end{align}
so taking the determinant of the inverse of the right-hand side
\begin{align}\det \{W(\psi_-),W(\psi_+)^{-1}\}=\det (I- \tilde \Gamma (f)\Gamma (g)).\end{align}
We also have
\begin{equation}W(\psi_-)W(\psi_+)^{-1}W(\psi_-)^{-1}W(\psi_+)=W(\psi_-)W(\psi_-\psi_+)^{-1}W(\psi_+)\end{equation}
so 
\begin{align}\det\bigl( W(\psi_-)W(\psi_+)^{-1}W(\psi_-)^{-1}W(\psi_+)\bigr)&=1/\det\bigl( W(\psi_+)^{-1}W(\psi_-\psi_+)W(\psi_-)^{-1}\bigr)\nonumber\\
&=1/\det\bigl( W(\psi_-)^{-1}W(\psi_+)^{-1}W(\psi_-\psi_+)\bigr)\nonumber\\
&=1/\det\bigl( W(\psi_-^{-1}\psi_+^{-1})W(\psi_-\psi_+)\bigr)\end{align}

\indent Taking the unitary conjugation by the Fourier transform, we have $\tilde \Gamma (f)\mapsto \Gamma_{\phi_1}$ and $\Gamma (g)\mapsto \Gamma_{\phi_2}$, where  
\begin{equation}\phi_1(x)= \int_{-\infty }^\infty \Bigl( {\frac{\psi_-(i\xi )}{\psi_+(i\xi )}}-1\Bigr)e^{-i\xi x}\, d\xi\end{equation}
\begin{equation}\phi_2(x)=\int_{-\infty }^\infty \Bigl( {\frac{\psi_+(-i\xi )}{\psi_-(-i\xi )}}-1\Bigr) e^{-i\xi x}\, d\xi;\end{equation}
\noindent the difference in signs $\pm\xi$ in the quotients reflecting the tilde on $\tilde \Gamma (f)$. \par
\indent Hence $\psi_-/\psi_+-1$ and  $\psi_+/\psi_--1$ belong to $L^2(i{\mathbb R})\cap L^\infty (i{\mathbb R})$ and determine bounded Hankel operators. We proceed to realise these via linear systems. Let $H=L^2(-\infty, \infty )$ and ${\mathcal D}(A)=\{ g\in H: \xi g(\xi )\in H\}$.  Then we introduce the linear systems $(-A,B_1,C)$ and $(-A,B_2, C)$ by 
\begin{align}\label{linsys2}A:{\mathcal D}(A)\rightarrow H: &\quad  g(\xi )\mapsto i\xi g(\xi)\qquad (g\in {\mathcal D}(A));\nonumber\\
B_1: {\mathbb C}\rightarrow H:&\quad \beta\mapsto \Bigl({\frac{\psi_-(i\xi )}{\psi_+ (i\xi )}}-1\Bigr) \beta \qquad (\beta\in {\mathbb C});\nonumber\\
B_2: {\mathbb C}\rightarrow H:&\quad \beta\mapsto \Bigl({\frac{\psi_+(-i\xi )}{\psi_- (-i\xi )}}-1\Bigr) \beta \qquad (\beta\in {\mathbb C});\nonumber\\
C: {\mathcal D}(A)\rightarrow {\mathbb C}: &\quad g\mapsto {\frac{1}{2\pi}}\int_{-\infty }^\infty g(\xi )\,d\xi \qquad (g\in {\mathcal D}(A)).\end{align}
\noindent Then $\phi_1 (t) =Ce^{-tA}B_1$ and $\phi_2(t) =Ce^{-tA}B_2$.  Also, $\Gamma_{\phi_1}$ and $\Gamma_{\phi_2}$ are Hilbert--Schmidt by Proposition \ref{Hom}. 
Hence $\Gamma_{\phi_1}\Gamma_{\phi_2}$ is a trace class operator, and $\det (I-\Gamma_{\phi_1}\Gamma_{\phi_2})$ is well defined.\par
\indent Suppose that the linear system $(-A,B_j,C)$ realises $\phi_j$. Then the matrix system
\begin{equation}\Bigl( \begin{bmatrix}-A&0\\ 0& -A\end{bmatrix} ,   \begin{bmatrix}B_1&0\\ 0& B_2\end{bmatrix} ,\begin{bmatrix}0&C\\ C& 0\end{bmatrix}\Bigr)\end{equation}
\noindent realises
\begin{equation}\Phi (x)= \begin{bmatrix}0&\phi_1(x)\\ \phi_2(x)& 0\end{bmatrix}\end{equation}
\noindent For finite matrices $U$ and $V$, we have
\begin{equation}\det \begin{bmatrix}I&0\\ 0&I-UV\end{bmatrix}=\det_2\begin{bmatrix}I&0\cr -V&I\end{bmatrix}\det_2\begin{bmatrix}I&U\\ V&I\end{bmatrix}\det_2\begin{bmatrix} I&-U\\ 0&I\end{bmatrix}\end{equation}
\noindent so by a simple approximation argument in Hilbert--Schmidt norm
\begin{equation}\det (I-\Gamma_{\phi_1}\Gamma_{\phi_2})=\det_2 (I+\Gamma_\Phi );\end{equation}
Hence 
\begin{equation}1/\det (W(\psi )W(\psi^{-1} )) =\det_2 (I+\Gamma_\Phi ).\end{equation}
\indent (i) Now by uniqueness of the Fourier transform, $\phi_1$ is real if and only if $\psi_-/\psi_+(i\xi )=\bar \psi_- (-i\xi )/\bar \psi_+ (-i\xi )$.\par
\indent (ii) Likewise $\phi_1(x)=\phi_2(x)$ if and only if $\psi_-(i\xi )/\psi_+(i\xi )=\psi_+(-i\xi )/\psi_-(-i\xi )$, which reduces to the stated condition. If $\phi_1=\phi_2$, then 
\begin{equation}\det (I+\lambda\Gamma_{\Phi}) =\det (I-\lambda^2\Gamma^2_{\phi_1})=\det_2(I-\lambda\Gamma_{\phi_1})\det_2(I+\lambda\Gamma_{\phi_1})\end{equation}
\noindent is determined by the spectrum of the scalar-valued Hankel operator $\Gamma_{\phi_1}$. The nature of the spectrum is determined in [20, 24].\par
\indent (iii) Evidently $\Gamma_\Phi$ is self-adjoint, if and only if $\bar\phi_1(x)=\phi_2(x)$; that is $\bar\psi_-(-i\xi )/\bar\psi_+(-i\xi) =\psi_+(-i\xi )/\psi_- (-i\xi )$. \par
 Finally, one considers the cases (i), (ii) and (iii) as they apply to $\begin{bmatrix}0&U\cr V&0\end{bmatrix}.$
\end{proof}
\noindent As in Corollary \ref{finitedet}, we can reduce the Fredholm determinant of Hankel operators to related determinants.  Let $P$ and $Q$ be orthogonal projections on $L^2(0, \infty )$ such that $P+Q=I.$ Then 
\begin{align}\det \bigl( P\{W(\psi_+)^{-1},W(\psi_-)\}P\bigr)&=
\det \{W(\psi_+)^{-1},W(\psi_-)\}\det \big( Q\{ W(\psi_-),W(\psi_+)^{-1}\} Q\bigr)\nonumber\\
&=\det (I- \tilde \Gamma (f)\Gamma (g))^{-1}\det \big( Q-Q \tilde \Gamma (f)\Gamma (g) Q\bigr).\end{align}
Self-adjoint block Hankel matrices have been characterized up to unitary equivalence, as in Theorem 2 of [20].
\par
\begin{cor}\label{corBarnes}Let $a_j,b_j,c_j,d_j\in (0, \infty )$  and
\begin{equation}\psi (i\xi )= \prod_{j=1}^m {\frac{\Gamma (a_j+i\xi)}{\Gamma (b_j+i\xi )}}\prod_{j=1}^\mu {\frac{\Gamma (c_j-i\xi)}{\Gamma (d_j-i\xi )}}\end{equation}
\noindent where the zeros and poles satisfy 
\begin{equation}\sum_{j=1}^m (a_j-b_j)=0= \sum_{j=1}^\mu (c_j-d_j).\end{equation}
\noindent Then there exists a linear system as in (\ref{linsys2}) such that
\begin{equation}1/\det (W(\psi )W(\psi^{-1} )) =\det_2 (I+\Gamma_\Phi ).\end{equation} 
\indent (i) Also, $\phi_1$ and $\phi_2$ are real.\par
\indent (ii) Suppose further that $m=\mu$ and $a_j=c_j$ and $b_j=d_j$ for $j=1, \dots, m$. Then $\phi_1=\phi_2$ and $\Gamma_\Phi$ is self-adjoint.\end{cor}
\vskip.05in
\begin{proof}(i) For $j=1, \dots, m$, let $a_j, b_j\in (0, \infty )$ be such that $\sum_{j=1}^m (a_j-b_j)=0$. 
Then
\begin{equation}\psi_-(i\xi )=\prod_{j=1}^m {\frac{\Gamma (a_j+i\xi)}{\Gamma (b_j+i\xi )}}\end{equation}
\noindent is meromorphic with poles at $\xi= ia_j, ia_j+i, ia_j+2i, \dots $ and zeros at  $\xi =ib_j, ib_j+i, ib_j+2i, \dots $, all in the open upper half plane. By considering the error term in Stirling's formula, one can prove that $\psi_-(i\xi )=1+O(1/\vert \xi\vert^{3/2})$ as $\xi\rightarrow\pm\infty$ along the real axis.\par
\indent Likewise, for $j=1, \dots, \mu$, let $c_j, d_j\in (0, \infty )$ be non zero real numbers such that $\sum_{j=1}^\mu (c_j-d_j)=0$. Then
\begin{equation}\psi_+(i\xi )=\prod_{j=1}^\mu {\frac{\Gamma (c_j-i\xi)}{\Gamma (d_j-i\xi )}}\end{equation}
\noindent is meromorphic with poles at $-ic_j, -ic_j-i, -ic_j-2i, \dots $ and zeros at  $-id_j, -id_j-i, -id_j-2i, \dots $, all in open lower half plane,  and $\psi_+(i\xi )=1+O(1/\vert \xi\vert^{3/2})$ as $\xi\rightarrow\pm\infty$ along the real axis.\par
 \indent Hence $\psi_-/\psi_+-1$ and  $\psi_+/\psi_--1$ belong to $L^1(i{\mathbb  R})$ and $L^\infty (i{\mathbb R})$ and determine bounded Hankel operators. We proceed to realise these via linear systems. Let $H=L^2(-\infty, \infty )$ and ${\mathcal D}(A)=\{ g\in H: \xi g(\xi )\in H\}$.  Then we introduce the linear systems $(-A,B_1,C)$ and $(-A,B_2, C)$ by 
\begin{align}A:{\mathcal D}(A)\rightarrow H: &\quad  g(\xi )\mapsto i\xi g(\xi)\qquad (g\in {\mathcal D}(A));\nonumber\\
B_1: {\mathbb C}\rightarrow H:&\quad \beta\mapsto \Bigl({\frac{\psi_-(i\xi )}{\psi_+ (i\xi )}}-1\Bigr) \beta \qquad (\beta\in {\mathbb C});\nonumber\\
B_2: {\nonumber C}\rightarrow H:&\quad \beta\mapsto \Bigl({\frac{\psi_+(-i\xi )}{\psi_- (-i\xi )}}-1\Bigr) \beta \qquad (\beta\in {\mathbb C});\nonumber\\
C: {\mathcal D}(A)\rightarrow {\mathbb C}: &\quad g\mapsto {\frac{1}{2\pi}}\int_{-\infty }^\infty g(\xi )\,d\xi \qquad (g\in {\mathcal D}(A)).\end{align}
\noindent Then $\phi_1 (t) =Ce^{-tA}B_1$ and $\phi_2(t) =Ce^{-tA}B_2$.  Also, $\Gamma_{\phi_1}$ and $\Gamma_{\phi_2}$ are Hilbert--Schmidt. Hence $\Gamma_{\phi_1}\Gamma_{\phi_2}$ is a trace class operator, and $\det (I-\Gamma_{\phi_1}\Gamma_{\phi_2})$ is well defined.\par
\indent (i) Here we have $\psi_- (i\xi )=\bar\psi_- (-i\xi )$ and  $\psi_+ (i\xi )=\bar\psi_+ (-i\xi )$, so $\phi_1$ and $\phi_2$ are real.\par  
\indent (ii) This is a special case of (ii) of the Theorem.\par
\end{proof}
\end{section} 
\begin{section}{Determinant expansions}\par
\indent In case (ii) of the Corollary \ref{corBarnes} we can compute $\phi_1$ and $\phi_2$ explicitly.  Theorem 1.4 page 237 of [24] shows that a Hankel operator is trace class if and only if it has a nuclear expansion as a series of Hankel operators of rank one. So to compute $\Gamma_{\phi_1}$ and  $\Gamma_{\phi_2}$ as trace class operators on $L^2(x, \infty )$, we select a sequence of exponential functions $(e^{-\lambda_jt})_{j=0}^\infty$ in $L^2(x, \infty )$ so that  $\Gamma_{\phi_1}$ has a nuclear expansion in terms of rank one Hankel operators; ultimately, this will enable us to compute the determinant of $I-\Gamma_{\phi_1}\Gamma_{\phi_2}$ compressed to $L^2(x, \infty )$ in terms of an infinite matrix. For large $x$, most of the entries of this matrix are very small, so this is a practicable means for computing the determinant. Our method follows [9].\par 
\indent  Let $s=i\xi$ and $z=e^{-x}$ where $\Re x>0$ so $\vert z\vert <1$. We consider
\begin{align}\phi_1(x)&={\frac{1}{2\pi}}\int_{-\infty}^\infty \Bigl( {\frac{\psi_-(i\xi )}{\psi_+(i\xi )}}-1\Bigr) e^{-i\xi x} \, d\xi\nonumber\\
 &={\frac{1}{2\pi i}}\int_{-i\infty}^{i\infty}\Bigl( \prod_{\ell =1}^m {\frac{\Gamma (a_\ell +s)  \Gamma (d_\ell -s)}{\Gamma (b_\ell +s)\Gamma (c_\ell -s)}}-1\Bigr)z^{-s}\, ds\nonumber\\
& ={\frac{1}{2\pi i}}\int_{-i\infty}^{i\infty} \Bigl( \prod_{\ell =1}^m      
{\frac{\Gamma (1-b_\ell -s){\hbox{cosec}\,} \pi (a_\ell +s)  \Gamma (d_\ell -s)} {\Gamma (1-a_\ell -s){\hbox{cosec}\,} \pi (b_\ell +s)\Gamma (c_\ell -s)}}-1\Bigr)z^{-s}\, ds\end{align}
\noindent where we have used the formula $\Gamma (w)\Gamma (1-w)=\pi {\hbox{cosec}}\, \pi w$; now we take an integral round a semicircular contour in the left half plane and sum over the residues at poles near the negative real axis of $s$ to obtain
\begin{align}\phi_1(x)&=\sum_{j=1}^m\sum_{k=0}^\infty{\hbox{Res}}(-a_j-k)\\
&= \sum_{j=1}^m \sum_{k=0}^\infty {\frac{(-1)^k}{\pi \Gamma (1+k)}}\prod_{\ell =1; \ell\neq j}^m {\frac{{\hbox{cosec}\,} \pi (a_\ell -a_j-k)}{\Gamma (1-a_\ell +a_j+k)}}\nonumber\\
&\quad\times  \prod_{\ell =1}^m {\frac{\Gamma (1-b_\ell +a_j +k)\Gamma (d_\ell +a_j+k)}{ {\hbox{cosec}\,} \pi (b_\ell -a_j-k)\Gamma (c_\ell +a_j+k)}} z^{a_j+k}
\end{align}
where we have picked out the factor ${\hbox{cosec}}\,\pi (s_j+s)/\Gamma (1-a_j-s)$ that contributes the pole, so
\begin{align}&\phi_1(x)\nonumber\\
&=\sum_{j=1}^m {\frac{1}{\pi}}\prod_{\ell : \ell\neq j} {\frac{ {\hbox{cosec}\,}\pi (a_\ell -a_j)}{\Gamma (1-a_\ell +a_j)}} \prod_{\ell =1}^m {\frac{\Gamma (1-b_\ell +a_j)\Gamma (d_\ell +a_j)}{ {\hbox{cosec}}\, \pi (b_\ell -a_j)\, \Gamma (c_\ell +a_j)}}\nonumber\\
&\quad \times  \sum_{k=0}^\infty \prod_{\ell =1}^m {\frac{(1-b_\ell +a_j)_k(d_\ell +a_j)_k}{(1-a_\ell +a_j)_k (c_\ell +a_j)_k}} z^{a_j+k},\cr
&=\sum_{j=1}^m {\frac{1}{\pi}}\prod_{\ell : \ell\neq j} {\frac{ {\hbox{cosec}\,}\pi (a_\ell -a_j)}{\Gamma (1-a_\ell +a_j)}} \prod_{\ell =1}^m  {\frac{\Gamma (1-b_\ell +a_j)\Gamma (d_\ell +a_j)}{ {\hbox{cosec}}\, \pi (b_\ell -a_j)\, \Gamma (c_\ell +a_j)}}z^{a_j}\nonumber\\
&\quad\times  {}_{2m}F_{2m-1}\Bigl[ {{1-b_1+a_j}\atop {1-a_1+a_j}}{{\dots}\atop{\dots}}  {{1-b_j+a_j}\atop {\hat 1}}{{\dots}\atop{\dots}} 
 {{1-b_m+a_j}\atop {1-a_m+a_j}} \,\,{{d_1+a_j}\atop {c_1+a_j}}{{\dots}\atop{\dots}} {{d_j+a_j}\atop {c_j+d_j}}{{\dots}\atop{\dots}}  
 {{d_m+a_j}\atop {c_m+a_j}}; z\Bigr],\nonumber\end{align}
\noindent where 
\begin{equation}z=e^{-x},\quad (c)_0=1, \quad (c)_k=c(c+1)\dots (c+k-1),\end{equation}
\noindent and $\hat 1$ stands for the omitted term in the denominator, and we have written this expression in terms of the generalized hypergeometric functions, as in [11], page 182. There is a similar formula for $\phi_2(z)$ in which $(c_j,d_j,b_j,a_j)$ replaces $(a_j,b_j,d_j,c_j)$\par
\indent Without loss of generality, we suppose $a_1<a_2<\dots <a_m$, so taking the term from ${\hbox{Res}}(-a_1)$, we have
\begin{equation}\phi_1(x)= {\frac{1}{\pi}}\prod_{\ell =2}^m {\frac{ {\hbox{cosec}\,}\pi (a_\ell -a_1)}{\Gamma (1-a_\ell +a_1)}} \prod_{\ell =1}^m {\frac{\Gamma (1-b_\ell +a_1)\Gamma (d_\ell +a_1)}{ {\hbox{cosec}}\, \pi (b_\ell -a_1)\, \Gamma (c_\ell +a_1)}} e^{-a_1x}+O(e^{-a_2x}+e^{-(a_1+1)x}),\end{equation}
\noindent and likewise with $c_1<c_2<\dots <c_m$, taking the term from ${\hbox{Res}}(-c_1)$, we have

\begin{equation}\phi_2(x)= {\frac{1}{\pi}}\prod_{\ell =2}^m {\frac{ {\hbox{cosec}\,}\pi (c_\ell -c_1)}{\Gamma (1-c_\ell +c_1)}} \prod_{\ell =1}^m {\frac{\Gamma (1-d_\ell +c_1)\Gamma (b_\ell +c_1)}{ {\hbox{cosec}}\, \pi (d_\ell -c_1)\, \Gamma (a_\ell +c_1)}} e^{-c_1x}+O(e^{-c_2x}+e^{-(c_1+1)x}).\end{equation}
\noindent We replace the doubly indexed family of powers by a singly indexed sequence by introducing $j=mk+r$ and  $\lambda_j=a_{r+1}+k$ and $\eta_j=c_{r+1}+k$, thus obtaining the sequences 
\begin{equation}\label{lambda}(\lambda_j)_{j=0}^\infty =(a_1, a_2, \dots ,a_m, a_1+1, a_2+1, \dots, a_m+1, a_1+2, \dots ),\end{equation}
\begin{equation}\label{eta}(\eta_j)_{j=0}^\infty =(c_1, c_2, \dots , c_m, c_1+1, c_2+1, \dots, c_m+1, c_1+2, \dots ),\end{equation}
\noindent where there is a recurring pattern of length $m$. With the coefficients given above, suitably re-indexed, let 
\begin{equation}\label{phi}\phi_1(x)=\sum_{j=0}^\infty \xi_j e^{-\lambda_jx}, \quad \phi_2(x)=\sum_{j=0}^\infty \gamma_j e^{-\eta_jx}.\end{equation}
\vskip.05in

\begin{prop}Suppose that $\phi_1$ and $\phi_2$ are as in (\ref{lambda}), (\ref{eta}) and (\ref{phi}). Then the determinant from Corollary \ref{corBarnes} is given by
\begin{equation}\det (I-\Gamma_{\phi_1}\Gamma_{\phi_2})\vert_{L^2(x, \infty )}=\det\begin{bmatrix} I&  \Bigl[ {\frac{e^{-x(\eta_\ell+\lambda_j )}\xi_j}{\eta_\ell+\lambda_j}}\Bigr]_{j,\ell =0}^\infty\\
\Bigl[ {\frac{e^{-x(\eta_j+\lambda_\ell )}\gamma_j}{\eta_j+\lambda_\ell }}\Bigr]_{j,\ell =0}^\infty&I\end{bmatrix}.\end{equation}\end{prop}
\vskip.05in
\begin{proof} We have a series of rank-one kernels
\begin{equation}\label{rankoneagain}\Gamma_{\phi_{1, x}}\leftrightarrow \sum_{j=0}^\infty \xi_j e^{-2\lambda_jx} e^{-\lambda_j(s+t)}\end{equation}
\noindent where $\sum \vert \xi_j\vert e^{-\lambda_jx}/\vert\lambda_j\vert$ converges, so $\Gamma_{\phi_1}$ is trace class on $L^2(x, \infty )$. 
 Then we introduce the linear systems $(-A_1, B_1, C_1)$ with ${\mathcal D}(A_1) ={\mathcal D}(A_2)=\{ (u_j)_{j=0}^\infty \in \ell^2: (ju_j)_{j=0}^\infty \in \ell^2\}$. 
\begin{align} B_1: {\mathbb C}\rightarrow \ell^2:&\quad \beta\mapsto (\xi_j)_{j=0}^\infty \beta ;\nonumber\\
                       A_1: {\mathcal D}(A_1)\rightarrow \ell^2: &\quad (u_j)_{j=0}^\infty \mapsto (\lambda_ju_j)_{j=0}^\infty;\nonumber\\
                       C_1: {\mathcal D}(A_1)\rightarrow {\mathbb C}:& \quad (u_j)\mapsto\sum_{j=0}^\infty u_j;\end{align}
\noindent and likewise  $(-A_2, B_2, C_2)$ 
\begin{align}B_2: {\mathbb C}\rightarrow \ell^2:&\quad \beta\mapsto (\gamma_j)_{j=0}^\infty \beta ;\nonumber\\
                       A_2: {\mathcal D}(A_2)\rightarrow \ell^2: &\quad (u_j)_{j=0}^\infty \mapsto (\eta_ju_j)_{j=0}^\infty ;\nonumber\\
                       C_2: {\mathcal D}(A_2)\rightarrow {\mathbb C}:& \quad (u_j)\mapsto\sum_{j=0}^\infty u_j;\end{align}
\noindent when we combine them into $(-A,B,C)$ 
\begin{equation}\Bigl( \begin{bmatrix}-A_1&0\cr 0&-A_2\end{bmatrix},  \begin{bmatrix}B_1&0\\ 0&B_2\end{bmatrix}, \begin{bmatrix}0&C_1\\ C_2&0\end{bmatrix}\Bigr)\end{equation}
\noindent with scattering function
\begin{align}\Phi (x)&=Ce^{-xA}B\nonumber\\
&=\begin{bmatrix}0&C_2e^{-xA_2}B_2\\ C_1e^{-xA_1}B_1&0\end{bmatrix}\nonumber\\
&=\begin{bmatrix}0&\phi_2(x)\\ \phi_1(x)&0\end{bmatrix},\end{align}
\noindent and write $\Phi (t+2x)=\Phi_{(x)}(t)$. We also consider the operator
\begin{align}\label{R}R_x&=\int_x^\infty e^{-tA}BCe^{-tA}\, dt\nonumber\\
&=\int_x^\infty \begin{bmatrix}e^{-A_1t}&0\\ 0&e^{-A_2t}\end{bmatrix}\begin{bmatrix}B_1&0\\ 0&B_2\end{bmatrix}\begin{bmatrix}0&C_1\\ C_2&0\end{bmatrix}\begin{bmatrix}e^{-A_1t}&0\\ 0&e^{-A_2t}\end{bmatrix} dt\nonumber\\
&=\begin{bmatrix} 0&\int_x^\infty e^{-A_1t}B_1C_2e^{-A_2t}dt\\ \int_x^\infty e^{-A_2t}B_2C_1e^{-tA_1}dt&0\end{bmatrix}.\end{align}
\indent To help compute the Fredholm determinant of $R_x$, we also let
$\Xi_{x} :L^2(0, \infty )\rightarrow \ell^2$ and $\Theta_x : L^2(0, \infty )\rightarrow \ell^2$ be defined by
\begin{equation}\Xi_x f=\int_x^\infty e^{-tA}B f(t)\, dt\end{equation}
\begin{equation}\Theta_x f=\int_x^\infty e^{-sA^\dagger} C^\dagger f(s)\, ds.\end{equation}
\noindent We observe that $\Theta_x$ is trace class, and likewise $\Xi_x$ is trace class since $\sum_{j}\xi_j e^{-\lambda_jx}$ converges absolutely. Whereas $(e^{-\lambda_jt})_{j=0}^\infty$ is not an orthogonal basis, the map $\Theta_x$ is injective by Lerch's theorem and ${\hbox{span}}\{ e^{-\lambda_jt} , j=0, 1, \dots \}$ is dense in $L^2(0, \infty )$.\par  
\indent Then we observe that $\Gamma_{\Phi_{(x)}}=\Theta_x^\dagger\Xi_x: L^2(0,\infty )\rightarrow L^2(0, \infty )$ and $R_{x}=\Xi_x\Theta_x^\dagger :\ell^2\rightarrow\ell^2$, so that 
\begin{equation}\det (I+R_{x})=\det (I+\Xi_x\Theta_x^\dagger) =\det (I+\Theta_x^\dagger \Xi_x )=\det (I+\Gamma_{\Phi_{(x)}} ).\end{equation}
\noindent hence
\begin{equation}\det (I+R_x)=\det (I-\Gamma_{\phi_1}\Gamma_{\phi_2})\vert_{L^2(0,\infty )}.\end{equation}
\noindent With respect to the standard orthonormal basis of $\ell^2$, we have a matrix representation
\begin{equation}R^{1,2}=\int_x^\infty e^{-A_1t}B_1C_2e^{-A_2t}dt=\Bigl[ {\frac{e^{-x(\eta_\ell+\lambda_j )}\xi_j}{\eta_\ell+\lambda_j }}\Bigr]_{j,\ell =0}^\infty\end{equation}
\noindent for the top right corner of $R_x$ as in (\ref{R}), and 
\begin{equation}R^{2,1}=\int_x^\infty e^{-A_2t}B_2C_1e^{-A_1t}dt=\Bigl[ {\frac{e^{-x(\eta_j+\lambda_\ell )}\gamma_j}{\eta_j+\lambda_\ell }}\Bigr]_{j,\ell =0}^\infty\end{equation}
\noindent for the bottom left corner of $R_x$ as in (\ref{R}).\par
\indent Whereas $R^{2,1}$ is not quite the transpose of $R^{1,2}$, the matrices have a high degree of symmetry which becomes clear when we make our expansion of the determinant. For a finite subset $S$ of ${\mathbb N}\cup \{0\}$, let $\sharp S$ be the cardinality of $S$, and $\Delta_{j\in S}(\lambda_j)$ be Vandermonde's determinant formed from $\lambda_j$ with $j\in S$ 
naturally ordered. For an infinite matrix $V$, and $T\subset  {\mathbb N}\cup \{0\}$, let $\det [V]_{S\times T}$ be the determinant formed from the submatrix of $V$ with rows indexed by $j\in S$ and columns indexed by $\ell\in T$, naturally ordered. Then 
\begin{align} \det (I-R^{1,2}R^{2,1})&=\sum_{S\subset {\mathbb N}\cup \{0\}} (-1)^{\sharp S} \det [R^{1,2}R^{2,1}]_{S\times S}\nonumber\\
&=\sum_{S,T\subset {\mathbb N}\cup \{0\}, \sharp S=\sharp T} (-1)^{\sharp S}\det [R^{1,2}]_{S\times T}\det [R^{2,1}]_{T\times S},\end{align}
\noindent by the Cauchy--Binet formula. Then by Cauchy's formula, the summand involving $S\times T$ is 
\begin{align}\label{Binet}\det [R^{1,2}]_{S\times T}&\det [R^{2,1}]_{T\times S}\nonumber\\
&=\exp\Bigl(-x\sum_{j\in S}\lambda_j-x\sum_{\ell\in T}\eta_\ell\Bigr){\frac{\prod_{j\in S}\xi_j\Delta_{j\in S}(\lambda_j)\Delta_{\ell\in T}(\eta_\ell)}{ \prod_{j\in S, \ell\in T}(\lambda_j+\eta_\ell )}}\nonumber\\
&\quad\times   
\exp\Bigl(-x\sum_{j\in T}\eta_j-x\sum_{\ell\in S}\lambda_\ell\Bigr){\frac{\prod_{j\in T}\gamma_j\Delta_{j\in T}(\eta_j)\Delta_{\ell\in S}(\lambda_\ell)}{ \prod_{\ell\in S, j\in T}(\lambda_\ell+\eta_j)}}\end{align}
\noindent Hence we have the determinant expansion 
of 
\begin{align} &\det(I-R^{1,2}R^{2,1}).\\
&=\sum_{S,T\subset {\bf N}\cup \{0\}, \sharp S=\sharp T} (-1)^{\sharp S}\exp\Bigl(-2x\sum_{j\in S}\lambda_j-2x\sum_{\ell\in T}\eta_\ell\Bigr)
{\frac{\prod_{j\in S}\xi_j\prod_{\ell\in T}\gamma_\ell \Delta_{j\in S}(\lambda_j)^2\Delta_{\ell\in T}(\eta_\ell )^2}{\prod_{j\in S, \ell\in T}(\lambda_j+\eta_\ell )^2}}.\nonumber\end{align}
\end{proof}
\indent We now make an approximation, similar to (2.30) from [2].  Suppose that $x$ is large, so that we only need retain the largest terms, which arise from $j=\ell =0$, that is $S=T=\{0\}$; then
\begin{align}\det (I+R_x)&=\det \begin{bmatrix}1& {\frac{e^{-x(\eta_0+\lambda_0)}\xi_0}{\eta_0+\lambda_0}}\\
{\frac{e^{-x(\eta_0+\lambda_0)}\gamma_0}{\eta_0+\lambda_0}}&1\end{bmatrix}\cr
&= 1-{\frac{e^{-2x(\lambda_0+\eta_0)}\xi_0\gamma_0}{(\eta_0+\lambda_0)^2}}\end{align}

 
\begin{defn}For $\Omega$ a domain in ${\mathbb C}$, a divisor is function $\delta: \Omega\rightarrow {\mathbb Z}$ such that $\{ z: \delta (z)\neq 0\}$ has no limit points in $\Omega$. In particular, the function $\delta_z: {\mathbb C}\rightarrow {\mathbb Z}$ given by $\delta_z(x)=1$ for $x=z$ and $\delta_z(x)=0$ for $x\neq z$ is a divisor.\end{defn}

The set of all divisors on $\Omega$ forms an additive group ${\mathcal D}(\Omega )$. For each meromorphic function, we associate the divisor given by the sum of $n\delta_z$ for each zero of order $n$ at $z$, and $-m\delta_p$ for each pole of order $m$ at $p$. For $\Gamma$-functions, it is convenient to have the following shorthand. For $s\in {\mathbb C}$ we write
\begin{equation} (s)_R=\delta_s+\delta_{s+1}+\delta_{s+2}+\dots,\end{equation}
\begin{equation} (s)_L=\delta_s+\delta_{s-1}+\delta_{s-2}+\dots.\end{equation}   
There is an additive subgroup ${\mathcal D}_\Gamma$ of ${\mathcal D}({\mathbb C})$ generated by the $\pm (a)_L$ and $\pm (b)_R$ with $a,b\in {\mathbb C}$, so that every $\delta \in {\mathcal D}_\Gamma$ arises from a quotient of products of Gamma functions, and ${\mathcal D}_\Gamma$ contains all finitely supported divisors.  
\end{section} 
\begin{section}{Examples}

\begin{ex} (i) Suppose that $\sum_{j=1}^m (-a_j+b_j)+\sum_{j=1}^{\mu}(-c_j+d_j)=0.$ Then there exists $\lambda$ such that $1-\lambda +\sum_{j=1}^m (-a_j+b_j)=0$ and $1-\lambda+\sum_{j=1}^\mu (c_j-d_j)=0$. Hence we can apply Corollary \ref{corBarnes} to $S(\xi\mid \lambda )=\psi (i\xi)=\tilde \psi_-(i\xi )\tilde \psi_+(i\xi )$  for the Wiener--Hopf factors
\begin{equation} \tilde \psi_-(i\xi )={\frac{\Gamma (\lambda -i\xi )}{ \Gamma (1-i\xi )}}\prod_{j=1}^m {{\Gamma (a_j+i\xi)}\over{\Gamma (b_j+i\xi )}}, \quad \tilde\psi_+(i\xi )= {\frac{\Gamma (1+i\xi )}{ \Gamma (\lambda +i\xi )}}\prod_{j=1}^\mu  {\frac{\Gamma (c_j-i\xi)}{\Gamma (d_j-i\xi )}}.\end{equation}
(ii) This example arises via the scattering amplitude in one-dimentional scattering theory. Let $q\in C_c^\infty ({\mathbb R}; {\mathbb R})$, and consider the Schr\"odinger equation with even potential $q$. There exist and even solution $f_+$ and and odd solution $f_-$ such that 

$$ -{\frac{d^2 f_{\pm }}{dx^2}}(x;\xi )+q(x) f_{\pm }(x; \xi )=\xi^2 f_{\pm }(x; \xi )$$
such that 
 $$f_{\pm }(x; \xi )\asymp e^{-i\xi x}-e^{i\xi x-i\theta_{\pm}(\xi )}\qquad (x\rightarrow\infty )$$
$$f_{\pm }(x; \xi )\asymp \pm\bigl(e^{-i\xi x}-e^{i\xi x-i\theta_{\pm}(\xi )}\bigr)\qquad (x\rightarrow-\infty )$$
so $\theta_{\pm}$ is the phase shift. Let the reflection coefficient be $R(\xi )=(-1/2) (e^{-i\theta_+(\xi )}+e^{-i\theta_-(\xi )})$. 
Then with $\phi (x)=(2\pi)^{-1}\int_{-\infty}^\infty e^{i\xi x}R(\xi )\, d\xi$, we introduce $\Gamma_\phi$ and 
$$\vartheta =\det (I+\Gamma_{\phi}),$$
\noindent as in [10]. In section 5.7 of [10], the authors interpret $\vartheta$ as a theta function on an infinite-dimensional torus, and obtain series expansions for the determinant. In the current paper, we use the exponential series (\ref{rankone}) and (\ref{rankoneagain}) instead, which lead to formulas (\ref{Binet}) which resemble those on 5.7 and 5.8 in [10].\par 
\indent In particular, consider the Schr\"odinger equation
\begin{equation}\label{Schrodinger}-f''(x)+{{\lambda (\lambda -1)}\over{\sinh^2 x}}f(x)=\xi^2 f(x).\end{equation}
\noindent Then the scattering amplitude is the coefficient of $f(x)$ for large $x$ when $e^{-i\lambda x}$ is the scattering function. Then 
\begin{equation}S(\xi \mid \lambda )=-e^{-i\theta_+(\xi )}=-{{\Gamma (1+i\xi )\Gamma (\lambda -i\xi )}\over{ \Gamma (1-i\xi )\Gamma (\lambda +i\xi )}}\end{equation}
\noindent which has divisor, in terms of $s=i\xi$, 
\begin{equation}-(-1)_L-(-\lambda)_R+(1)_R+(-\lambda )_L.\end{equation}
\end{ex}
\vskip.05in

\begin{ex} Meier's $G$-function [11] page 206  is 
\begin{equation}G^{m,n}_{p,q}\Bigl( x\Bigl\vert {{a_1, \dots, a_p}\atop{b_1, \dots, b_q}}\Bigr) ={\frac{1}{2\pi i}}\int_{\gamma -i\infty}^{\gamma +i\infty}{\frac{\prod_{j =1}^m\Gamma (b_j-s)\prod_{j =1}^n\Gamma (1-a_j+s)}{\prod_{j =m+1}^q\Gamma (1-b_j+s)\prod_{j =n+1}^p\Gamma (a_j-s)}} x^s\, ds\end{equation}
\noindent where we take all the $a_j, b_j$ in $\{ s: 0< \Re s<1\}$ with degree $p+q-2m-2n$, which we take to be negative. Then the divisor for the quotient of Gamma functions in the integrand is 
\begin{equation} \sum_{j=1}^m -(b_j)_R-\sum_{j=1}^n (a_j-1)_L+\sum_{j=m+1}^q (b_j-1)_L+\sum_{j=n+1}^p(a_j)_R.\end{equation}
\noindent Then the integral converges for $\vert \arg x\vert< (2m+2n-p-q)\pi/2$. \par
\indent One can express various applications of Corollary 4.3 in terms of $G$.\par
\end{ex}
\vskip.05in

\begin{ex} Hankel matrices also arise from functions on the finite-dimensional real torus. Let $1/2<\nu <1$ and observe that Struve's function $S_\nu$ [28, page 127] has Mellin transform
\begin{equation}S_\nu^*(s) =2^{s-1} {\frac{\tan((\pi/2 )(s+\nu ) )\Gamma ((s+\nu )/2)}{\Gamma ((\nu-s+2)/2)}},\end{equation}
\noindent which is holomorphic on $-\nu <\Re s<1-\nu$; see [28]. Also, for $s=\eta +i\xi$ and $-1<\eta<-1/2$, we have $S^*_\nu(\eta+i\xi )=O(\vert\xi\vert^\eta)$ as $\xi\rightarrow\pm\infty$, so $S^*(\eta +i\xi )\rightarrow 0$ as $\xi\rightarrow\pm\infty$ and
\begin{equation}\int_{-\infty}^\infty \vert S_\nu^* (\eta +i\xi )\vert^2\, d\xi<\infty,\end{equation}
\noindent hence by Plancherel's formula, we have
\begin{equation}\int_0^\infty x^{2\eta -1}S_\nu (x)^2\, dx<\infty.\end{equation} 
We have the determinant of the finite Hankel matrix
\begin{align}I_n(t)&={\frac{\pi^{n^2/2}}{t^{n(n-1)}2^{n(n-1)} }}\det\Bigl[ \Gamma (j+k+1/2) S_{j+k}(t)\Bigr]_{j,k=0, \dots, n-1}.\\
&=\det\Bigl[{\frac{S_{j+k}(t) \Gamma (j+k+1/2)\sqrt{\pi}}{2^{j+k}t^{j+k}}}\Bigr]_{j,k=0, \dots , n-1}\nonumber\\
&=\det\Bigl[\int_{[0,\pi /2] } \sin^{2(j+\ell )}\theta \, \sin(t\cos \theta )\, d\theta\Bigr]_{j, k =0, \dots , n-1}\nonumber\\
&={\frac{1}{n!}}\int_{[0, \pi/2]^n} \det [\sin^{2k }\theta_j]_{j,k =0, \dots , n-1}^2\prod_{j=0}^{n-1} \sin (t\cos\theta_j)\, d\theta_j\nonumber\\
&={\frac{1}{n!}}\int_{[0, \pi/2]^n} \prod_{j,k=0, \dots, n-1; j< k} (\sin^2\theta_j-\sin^2\theta_k )^2\prod_{j=0}^{n-1} \sin (t\cos\theta_j)\, d\theta_j\end{align} 
\noindent where the final formula resembles the Weyl integration formula for a class function on the symplectic group $Sp(n)$.\par
\end{ex}
\vskip.05in
\noindent The following example gives a case in which moments satisfy a type of recurrence relation, but do not quite satisfy the conclusions of the Theorem (\ref{Hankelfactors}). The linear system representation is found explicitly.\par
\vskip.05in
\begin{prop} For $\kappa >1$, introduce the weight $w(x)=\log (2\kappa /(1-x))$ for $-1<x<1$. Then the moment matrix
$[\mu_{j+k}]_{j,k=0}^\infty$ defines a bounded linear operator on $\ell^2$ which is not Hilbert--Schmidt.\end{prop}
\begin{proof} Here $w(x)\rightarrow \infty$ as $x\rightarrow 1-$, and $w(x)\rightarrow\log \kappa $ as $x\rightarrow (-1)+$. The moments satisfy
\begin{equation}\mu_0=\int_{-1}^1 w(x)\, dx=2\log (2\kappa )+2-2\log 2,\end{equation}
and generally
\begin{equation}\mu_n=\int_{-1}^1 x^nw(x)\, dx=\log (2\kappa )\int_{-1}^1 x^n\, dx-\int_{-1}^1 x^n\log (1-x)\, dx,\end{equation}
where by integrating by parts, one obtains
\begin{align}\int_{-1}^1 &-x^n\log (1-x)\, dx\nonumber\\
&=\bigl[x^n(1-x)\log (1-x)]_{-1}^1+\int_{-1}^1 x^ndx-n\int_{-1}^1x^{n-1}(1-x)\log (1-x)\, dx\nonumber\\
&=-(-1)^n2\log 2+{\frac{1-(-1)^{n+1}}{n+1}}-n\int_{-1}^n x^{n-1}\log (1-x)\, dx+n\int_{-1}^1 x^n\log (1-x)\, dx,\end{align}
so by solving the recurrence relation
\begin{equation}-(n+1)\int_{-1}^1 x^n\log (1-x)\, dx=-n\int_{-1}^1 x^{n-1}\log (1-x)\, dx-(-1)^n2\log 2+{\frac{1-(-1)^{n+1}}{n}},\end{equation}
we obtain 
\begin{equation}-(n+1)\int_{-1}^1x^n\log (1-x)\, dx=-\int_{-1}^1\log (1-x)\, dx-2\log 2\sum_{m=1}^n (-1)^m +\sum_{m=1}^n{\frac{1-(-1)^{m+1}}{m+1}};\end{equation}
hence 
\begin{align}\int_{-1}^1 x^nw(x)\, dx&={\frac{1-(-1)^{n+1}}{n+1}}\log (2\kappa )\nonumber\\
&\quad +{\frac{1}{n+1}}\Bigl( 2-2\log 2+(1+(-1)^{n+1})\log 2+\sum_{m=1}^n{\frac{1-(-1)^{m+1}}{m+1}}\Bigr);\end{align}
the final term includes the sum
\begin{equation}\sum_{m=1}^n{\frac{1-(-1)^{m+1}}{m+1}} =\log n+\log 2+\gamma +o(1)\qquad (n\rightarrow\infty ),\end{equation}
\noindent where here $\gamma$ is Euler's constant. Hence the Cauchy transform
\begin{equation}G(z)=\int_{-1}^1 {\frac{w(t)\, dt}{z-t}}=\sum_{n=0}^\infty {\frac{\mu_n}{z^{n+1}}}\end{equation}
\noindent diverges at some points with $\vert z\vert =1$. Also $\sum_{j=0}^\infty j\mu_{j}^2$ diverges, so the Hankel moment matrix is not Hilbert--Schmidt.\par
\indent Nevertheless, the Hankel moment matrix $[\mu_{j+m}]_{j,m=0}^\infty$ defines a bounded linear operator on $\ell^2$. To see this, we transform to Hankel integral operators on $L^2(0, \infty )$ via the Laguerre functions. With $J_0$ standing for Bessel's function of the first kind of order zero, the orthonormal Laguerre functions in $L^2(0, \infty )$ satisfy   
$$e^{-x/2}L_n(x)={\frac{e^{x/2}}{n!}}\int_0^\infty t^nJ_0(2\sqrt{xt}) e^{-t}\, dt.$$
We introduce the scattering function
$$\phi(x)=\sum_{n=0}^\infty \mu_ne^{-x/2}L_n(x)$$
which we can express as an integral
\begin{align}\phi(x)&=e^{x/2}\sum_{n=0}^\infty \int_0^\infty {\frac{t^n\mu_n}{n!}}J_0(2\sqrt{xt}) e^{-t}\, dt\\
&=e^{x/2}\sum_{n=0}^\infty \int_0^\infty\int_{-1}^1 v^nw(v)\, dv {\frac{t^n}{n!}}J_0(2\sqrt{xt}) e^{-t}\, dt\\
&=e^{x/2}\int_0^\infty\int_{-1}^1\sum_{n=0}^\infty  v^n{\frac{t^n}{n!}}w(v)\, dv J_0(2\sqrt{xt}) e^{-t}\, dt\\
&=e^{x/2}\int_{-1}^1 \int_0^\infty e^{vt}J_0(2\sqrt{xt})e^{-t}\, dt w(v)\, dv.\end{align}
By Webber's integral [GR, 6.63(4); E ITv1,4.14(25), p. 185], the inside integral is 
\begin{equation}\int_0^\infty e^{-t+vt}J_0(2\sqrt{xt})\, dt={\frac{1}{1-v}}e^{-x/(1-v)},\end{equation}
so we have, on substituting $s=1/(1-v)$,
\begin{align} \phi (x)&=e^{x/2}\int_{-1}^1 w(v)e^{-x/(1-v)}\, {\frac{dv}{1-v}}\\
&=e^{x/2}\int_{1/2}^\infty\log (2\kappa s) e^{-sx}{\frac{ds}{s}}.\end{align}
It is convenient to introduce the incomplete Gamma function
\begin{equation}\phi_\nu (x)=e^{x/2}\int_{1/2}^\infty s^{\nu-1} e^{-xs}\, ds\end{equation}
and write
\begin{equation}\phi (x)=\log (2\kappa )\phi_0(x)+\Bigl({\frac{\partial }{\partial \nu}}\Bigr)_{\nu=0}\phi_\nu (x)\end{equation}
\indent We now express $\phi_\nu$ as the scattering function of a continuous time linear system. Let the state space be $H=L^2(0, \infty )$, with dense linear subspace ${\mathcal D}(A)=\{ f\in H: tf(t)\in H\}$. Then for $-1/2\leq \Re\nu < 1/2$, we introduce the linear system $(-A,B_\nu ,C)$ by
\begin{align} A: {\mathcal D}(A)\rightarrow H:&\qquad f\mapsto tf(t)\qquad (f\in {\mathcal D}(A))\\
B_\nu:{\mathbb C}\rightarrow H: &\qquad b\mapsto b(1/2+t)^{\nu-1}b\qquad (b\in {\mathbb C})\\
C: {\mathcal D}(A)\rightarrow {\mathbb C}: &\qquad f\mapsto \int_0^\infty f(t)\, dt\qquad (f\in {\mathcal D}(A));\end{align}
the corresponding scattering function is
\begin{equation}\phi_\nu (x)=Ce^{-xA}B_\nu=\int_0^\infty (1/2+t)^{\nu-1} e^{-xt}\, dt.\end{equation}
Then we introduce the operator
\begin{equation}R_\nu =\int_0^\infty e^{-xA}B_\nu Ce^{-xA}\, dx,\end{equation}
which is the integral operator on $L^2(0,\infty )$ that has kernel
\begin{equation}R_\nu (\tau, t)={\frac{(1/2+\tau)^\nu}{(1/2+\tau )}}{\frac{1}{t+\tau }}.\end{equation}
Evidently $R_\nu$ is the composition of Hilbert's Hankel operator with kernel $1/(\tau +t)$ and multiplication by $(1/2+\tau )^{\nu -1},$ so $R_\nu$ is bounded on $L^2(0, \infty ).$ Operators of this form were considered by Howland [16].
\end{proof}
\end{section}
\begin{section}{Application of equilibrium problem to linear statistics}
In this section we consider
\begin{equation}{\mathbb E} e^{-\sum f}={\frac{\int_{(0,b)^n} \exp (-\sum_{j=1}^n f(x_j)) \prod_{1\leq j<k\leq n} (x_j-x_k)^2\prod_{j=1}^n w_0(x_j)dx_j}{
\int_{(0,b)^n} \prod_{1\leq j<k\leq n} (x_j-x_k)^2\prod_{j=1}^n w_0(x_j)dx_j}}\end{equation} 
with particular emphasis on $w_0=e^{-nv_0}$ where $v_0\in C^2$ is convex and $f(x)=\vartheta (x-t)$ is a step function. Then the exponent in numerator of the expression (\ref{linstat}) involves
\begin{equation} {\frac{1}{n^2}}\sum_{j=1}^n f(x_j)+ {\frac{1}{n}}\sum_{j=1}^n v_0(x_j)-{\frac{1}{n^2}}\sum_{1\leq j,k\leq n:j\neq k} \log\vert x_j-x_k\vert.\end{equation}
We regard this as the electrostatic energy associated with $n$ positive and equal charges on a line, subject to an electrical field.  The following result [26] extends a familiar result to the case of discontinuous fields.\par

\begin{lem}{Let $\Sigma $ be a closed subset of the Riemann sphere and let $v: \Sigma\rightarrow {\Rb}\cup \{ \pm \infty\}$ be lower semi continuous, $v<\infty$ on a set of positive logarithmic capacity and suppose that there exists $c>0$ such that $v(z)\geq c\log \vert z\vert$ as $z\rightarrow \infty$ for $z\in \Sigma $. Then the minimization problem in the collection of all probability measures on $\Sigma$,
\begin{equation} \inf_\sigma\Bigl\{ \int_\Sigma v(x)\sigma (dx)+\int\!\!\!\int_{\Sigma\times \Sigma} \log {\frac{1}{ \vert x-y\vert}}\, \sigma (dx)\sigma (dy)\Bigr\}\end{equation}
has a unique minimizer $\sigma$, with support $S\subseteq \Sigma$. Furthermore, there exists $C\in \Rb$ such that 
\[v(x)=2\int_S\log \vert x-y\vert \sigma (dy)+C\]
for quasi almost all $x$ in $S$.}\end{lem}

\indent Let $v_0$ be $C^2$ and convex, with $v_0(x)\geq c\log (1/x)$  as $x\rightarrow 0+$ and $v_0(x)\geq c\log x$ as $x\rightarrow\infty$ for some $c>0$. Then there exists a probability measure $\sigma_0$ supported on $[a,b]\subset (0, \infty )$ and  constant $C$ such that 
$$v_0(x)\geq 2\int_a^b \log \vert x-y\vert \sigma_0(dy)+C$$
with equality for all $x\in (a,b)$. We replace the weight $w_0$ by $w$, the potential $v_0$ by $v=v_0+\beta f$, hence $\sigma_0$ by $\sigma=\sigma_0+\rho$, where $\int \rho=0$, and consider the integral equation  
\begin{equation}\label{defnrho}f(x)=2\int_a^b \log \vert x-y\vert \rho (y)dy+c_1.\end{equation}
The probability density of the linear statistic has mean
\begin{equation}\label{S2} S_2=\int_a^b f(x)\sigma_0(x)dx\end{equation}
and variance
\begin{equation}\label{S1}S_1={\frac{1}{4\pi^2}}\int_a^b\int_a^b {\frac{f(x)}{\sqrt{(x-a)(b-x)}}}{\frac{\partial }{\partial y}}\Bigl({\frac{\sqrt{(b-y)(y-a)}}{x-y}}\Bigr)\, f(y)\, dydx.\end{equation} 

To compute $S_2$, one uses Fourier series. 
\begin{lem}\label{Fourier} For $n\in {\mathbb N}$, let
$(\pi /2)a_n=\int_0^{\pi}v_0( {\frac{a+b}{2}}+{\frac{b-a}{2}}\cos\theta )\cos n\theta d\theta$. Suppose that $(na_n)\in \ell^2$, and $f$ is bounded. Then
\begin{equation}S_2=\int_0^{\pi}f\Bigl( {\frac{a+b}{2}}+{\frac{b-a}{2}}\cos\theta \Bigr){\frac{d\theta}{\pi}}-\sum_{n=1}^\infty na_n\int_0^{\pi}f\Bigl( {\frac{a+b}{2}}+{\frac{b-a}{2}}\cos\theta \Bigr)\cos n\theta\, d\theta.\end{equation}
\end{lem}

\begin{proof} We write 
\begin{equation}v_0\Bigl( {\frac{a+b}{2}}+{\frac{b-a}{2}}\cos\theta \Bigr) =\sum_{n=0}^\infty a_n\cos n\theta.\end{equation}
Then for $n\in {\mathbb N}$, we let
\begin{equation}h_n(\phi )=\int_0^\pi \log\vert \cos\phi-\cos\theta\vert \cos n\theta \, d\theta\end{equation}
so
\begin{equation}h_n'(\phi )=-\sin\phi \,{\hbox{p.v.}}\int_0^\pi{\frac{\cos n\theta}{\cos\phi -\cos \theta}} d\theta =\sin n\phi.\end{equation}
For $n$ even,  $h_n(\pi /2)$ is given by
\begin{equation}\int_0^\pi \log\vert\cos\theta\vert \cos (2k\theta )\, d\theta=2^{-1}\int_0^{2\pi} \log\vert \cos (x/2)\vert \cos kx\, dx={\frac{(-1)^k\pi}{2k}},\end{equation}
whereas for $n$ odd, $h_n(\pi /2)$ is given by
\begin{equation}\int_0^\pi \log\vert\cos\theta\vert \cos (2k+1)\theta\, d\theta=0;\end{equation}
so 
\begin{equation}h_n(\pi/2)-h_n(\phi)=\int_\phi^{\pi/2}\sin n\psi\, d\psi =-{\frac{\cos n\pi/2}{n}}+{\frac{\cos n\phi}{n}},\end{equation}
so $h_n(\phi )={\frac{-1}{n}}\cos n\phi$. We can therefore write a solution of the extremal problem as 
\begin{equation}\sigma_0 \Bigl( {\frac{a+b}{2}}+{\frac{b-a}{2}}\cos\theta \Bigr) \vert\sin \theta \vert ={\frac{1}{\pi}}-\sum_{n=1}^\infty na_n\cos n\theta\qquad (0<\theta <\pi )\end{equation}
which by hypothesis is $L^2$ convergent, and then we establish the equality
\begin{align}S_2 &=\int_0^{\pi}f\Bigl( {\frac{a+b}{2}}+{\frac{b-a}{2}}\cos\theta \Bigr)\sigma_0 \Bigl( {\frac{a+b}{2}}+{\frac{b-a}{2}}\cos\theta \Bigr)\sin\theta \, d\theta\nonumber\\
&=\int_0^{\pi}f\Bigl( {\frac{a+b}{2}}+{\frac{b-a}{2}}\cos\theta \Bigr){\frac{d\theta}{\pi}}-\sum_{n=1}^\infty na_n\int_0^{\pi}f\Bigl( {\frac{a+b}{2}}+{\frac{b-a}{2}}\cos\theta \Bigr)\cos n\theta\, d\theta\end{align} 
\end{proof}

\begin{ex} (i) In the context of (\ref{Dtheta}) Let $[a,b]=[-1,1]$ and for $t\in [-1,1]$ let $f(x)=\vartheta (x-t)$, 
and let $U_n$ is the Chebyshev polynomial of the second kind of degree $n$ such that 
$U_n(\cos\phi )=\sin (n+1)\phi/\sin\phi$.
Then by the Lemma \ref{Fourier}, we have
$$S_2={\frac{\cos^{-1} t}{\pi}}-\sum_{n=1}^\infty \sqrt{1-t^2}U_{n-1}(t)(2/\pi )\int_0^\pi v(\cos \theta)\cos n\theta \, d\theta.$$ 
(ii) In the context of (\ref{newweight})Let $f(x)=\pi^{-1} \tan^{-1} ((x-t)/\varepsilon )$, and consider $-1<t<1$ for $[a,b]=[-1,1]$;  with $z_{\pm}=t\pm i\varepsilon$. Then by considering the integrals 
\begin{equation} {\hbox{p.v.}}\int_{-1}^1 {\frac{\sqrt{1-y^2}}{y-x}}{\frac{dy}{y-z}}=-\pi +{\frac{\pi\sqrt{z^2-1}}{z-x}}\qquad (z\in {\mathbb C}\setminus [-1,1])\end{equation}
where the branch of the square root is chosen so that the integrals converge to zero as $z\rightarrow\infty$, we deduce that $\rho$ from (\ref{defnrho}) satisfies 
\begin{equation} \rho (x)={\frac{\sqrt{1-x^2}}{2\pi^2} }{\frac{1}{2i}}\Bigl( {\frac{-1}{(x-z_+)\sqrt{z_+^2-1}}} +{\frac{1}{(x-z_-)\sqrt{z_-^2-1}}}\Bigr) +{\frac{C}{\sqrt{1-x^2}}},\end{equation}
where $C$ is chosen so that $\int_{-1}^1 \rho (x)\, dx=0.$ As we cross $[-1,1]$, the square root changes sign.\par

\end{ex}

\begin{ex} For $f\in L^\infty$, we consider
\begin{equation}f(x)=\int_0^\infty \log\vert x-y\vert \rho (y)\, dy+c_1\end{equation}
so
\begin{equation}f(0)=\int_0^\infty\log \vert y\vert\rho(y)\, dy+c_1\end{equation}
and subtracting, we have a Mellin convolution
\begin{equation}f(x)-f(0)=\int_0^\infty \log \Bigl\vert {\frac{x}{y}}-1\Bigr\vert y\rho (y){\frac{dy}{y}},\end{equation}
so
\begin{equation}M(f(x)-f(0); s)=-{\frac{\pi}{s}}\tan \pi (s+1/2)M(x\rho  (x);s)\end{equation}
so 
\begin{equation}M(x\rho (x); s)=s^2\Bigl( -{\frac{\tan \pi s}{\pi s}}\Bigr)M(f(x)-f(0);s).\end{equation}
We deduce that 
\begin{equation}x\rho (x)={\frac{1}{\pi^2}}\Bigl( x{\frac{d}{dx}}\Bigr)^2\int_0^\infty 
y(f(y)-f(0)) \log {\frac{\sqrt{x}+\sqrt{y}}{\vert\sqrt{x}-{\sqrt{y}}\vert}}\, {\frac{dy}{y}}.\end{equation}
Now we let 
\begin{equation}g_2(x)= \Bigl( x{\frac{d}{dx}}\Bigr)^2\log {\frac{\sqrt{x}+1}{\vert\sqrt{x}-1\vert}}= {\frac{\sqrt{x}(x+1)}{2(x-1)^2}}\end{equation}
which has Mellin transform
\begin{equation}M(g_2; s)= {\frac{-s\tan \pi s}{\pi }}.\end{equation}
Hence by the Plancherel formula for the Mellin transform
\begin{align}\int_0^\infty xf(x) (g_2\ast f)(x)dx&={\frac{1}{2\pi i}}\int_{-i\infty}^{i\infty}  M(g_2, s)M(f;s)M(f;-s)ds
\\
&={\frac{1}{2\pi^2}}\int_{-\infty}^\infty \xi \tanh (\pi \xi) M(f;i\xi )M(f; -i\xi )\, d\xi.\end{align}
One compares this formula with Proposition \ref{commutingdiagram}.

\end{ex}

\end{section}

\begin{section}{References}

\vskip.05in
\noindent [1] E.L. Basor, Distribution functions for random variables for ensembles of positive Hermitian matrices, {\sl Comm. Math. Phys.} {\bf 188} (1997), 327-350.\par
\noindent [2] E.L. Basor and Y. Chen,  A note on Wiener--Hopf determinants and the Borodin--Okounkov identity, {\sl Integral Equations and Operator Theory}, {\bf 45} (2003), 301-308.\par
\noindent [3] E.L. Basor and Y. Chen, Toeplitz determinants from compatibility conditions, {\sl Ramanujan Math. J.} {\bf 16} (2008), 25-40.\par 
\noindent [4] E.L. Basor, Y. Chen and H. Widom, Hankel determinants and Fredholm determinants, pp 21-29, {\sl Random matrix models and their applications} , (MSRI Publications, Cambridge University Press).\par
\noindent [5] E.L. Basor and T. Ehrhardt, Asymptotics of determinants of Bessel operators, {\sl Comm. Math. Phys.} {\bf 234} (2003), 491-516.\par
\noindent [6] E.L. Basor and C.A. Tracy, Variance calculations and the Bessel kernel, {\sl J. Statistical Physics}  {\bf 73}, 415-421.\par
\noindent [7] Y. Chen and A.R. Its, A Riemann--Hilbert approach to the Akhiezer polynomials, {\sl Phil. Trans. R. Soc. A} {\bf 366} (2008), 973-1003.\par
\noindent [8] Y. Chen and N. Lawrence, On the linear statistics of Hermitian random matrices, {\sl J. Phys. A} {\bf 31} (1998), 1141-1152.\par
\noindent [9] G. Blower, On linear systems and $\tau$-functions associated with Lam\'e's equation and Painlev\'e's equation VI, {\sl J. Math. Anal. Appl.} (2011), 294-316.\par
\noindent [10] N. Ercolani and H.P. McKean, Geometry of KDV(4): Abel sums, Jacobi variety and theta functions in the scattering case, {\sl Invent. Math.} {\bf 99} (1990), 483-544.\par 
\noindent [11] A. Erdelyi, {\sl Higher Transcendental Functions} Volume 1\par

\noindent [12] A.S. Fokas, A.R. Its, A.A. Kapaev and V.Y. Novokshenov, {\sl Painlev\'e Transcendents. The Riemann-Hilbert approach.} (American Mathematical Society, 2006).\par 
\noindent [13] E. Heine, {\sl Handbuch der Kugelfunctionen: Theorie und Anwendungen}, (Physica Verlag, W\"urzbug, 1961)\par 
\noindent [14] J.W. Helton and R.E. Howe, Integral operators, commutators, traces, index and homology, pp 141-209 in Lecture Notes in Mathematics 345, Springer, 1973.\par 
\noindent [15] N. Hitchin, A lecture on the octahedron, {\sl Bull. London Math. Soc} {\bf 35} (2003), 577-600.\par
\noindent [16] J.S. Howland, Spectral theory of operators of Hankel type II, {\sl Indiana Univ. Math. J.} {\bf 41} (1992), 427-434.\par
\noindent [17] M. Jimbo and T. Miwa, Monodromy preserving deformation of linear ordinary differential equations with rational coefficients. II. {\sl Physica D} {\bf 2} (1981), 407-448.\par  
\noindent [18] A.V. Kitaev and D.A. Korotkin, Solutions of the Schlesinger equations in terms of $\Theta$-functions, {\sl Internat. Math. Res. Notices} {\bf 17} 1998, 877-905.\par 
\noindent [19] A.P. Magnus, Painlev\'e type differential equations for the recurrence coefficients of semi-classical orthogonal polynomials, {\sl J. Comput. Appl. Math.} {\bf 57} (1995), 215-237.\par
\noindent [20] A.V. Megretskii, V.V. Peller, and S.R. Treil, The inverse spectral problem for self-adjoint Hankel operators,  {\sl Acta Math.} {\bf 174} (1995), 241-309.\par
\noindent [21] C. Min and Y. Chen, Gap probability distribution of the Jacobi unitary ensemble: elementary treatment, from finite $n$ to double scaling, {\sl Stud. Appl. Math.} {\bf 140} (2018), 202-220.\par
\noindent [22] J. Palmer, Deformation analysis of matrix models, {\sl Physica D} {\bf 78} (1994), 166-185.\par
\noindent [23] J.R. Partington and G. Weiss, Admissible observation operators for the right-shift semigroup, {\sl Math. Control Signals Systems} {\bf 13} (2000), 179-192.\par 
\noindent [24] V.V. Peller, {\sl Hankel Operators and their Applications}, (Springer, 2003).\par 
\noindent [25] M. van der Put and M.F. Singer, {\sl Galois theory of linear differential equations}, (Springer, 2003).\par 
\noindent [26] P. Simenonov, A weighted energy problem for a class of admissible weights, {\sl Houston J. Math.} {\bf 31} (2005), 1245-1260.\par
\noindent [27] V.I. Smirnov,  {\sl Integral Equations and Partial Differential Equations}, volume IV, (Pergamon, 1964).\par
\noindent [28] I.N. Sneddon, {\sl The use of integral transforms}, (McGraw--Hill, 1972).\par 
\noindent [29] J.L. Taylor, Banach algebras and topology pp 118--186 in {\sl Algebras in Analysis} J.R. Williamson (edr), (Academic Press, 1975).\par
\noindent [30] E.C. Titchmarsh, {\sl Introduction to the theory of Fourier integrals} (third edition) (Oxford University Press, 1986).\par

\noindent [31] C.A. Tracy and H. Widom, Level spacing distributions and the Bessel kernel, {\sl Comm. Math. Phys.} {\bf 161} (1994), 289--309.\par
\noindent [32] C.A. Tracy and H. Widom, Fredholm determinants, differential equations and matrix models, {\sl Comm. Math. Phys.} {\bf 163} (1994), 33-72.\par
\noindent [33] K. Trim\'eche, Transformation int\'egrale de Weyl et th\'eor\`eme de Paley-Wiener associ\'es \`a un op\'erateur diff\'erentiel singulier sur $(0,\infty )$, {\sl J. Math. Pures Appl.} {\bf 60} (1981), 51-98.\par
\noindent [34] J. Wermer, {\sl Banach Algebras and Several Complex Variables}, (second edition) (Springer, 1976).\par
\end{section}
\end{document}